\documentclass[11pt,reqno]{amsart}

\usepackage{comment}
\usepackage[margin=1.3in]{geometry}
\usepackage{calc, xcolor, mathrsfs, mathtools, enumitem}

\usepackage{tikz, amsfonts, amsmath, amsthm, amssymb, physics}

\usepackage{stackengine}
\stackMath
\newcommand\ntilde[2][2]{%
 \def\useanchorwidth{T}%
  \ifnum#1>1%
    \stackon[-.5pt]{\ntilde[\numexpr#1-1\relax]{#2}}{\scriptscriptstyle\sim}%
  \else%
    \stackon[.5pt]{#2}{\scriptscriptstyle\sim}%
  \fi%
}
\newcommand{\stilde}{\ntilde[1]} 
\newcommand{\dtilde}{\ntilde[2]} 
\DeclareMathOperator{\Image}{Im} 

\usepackage{hyperref}

\theoremstyle{definition}
\newtheorem{thm}{Theorem}[section]
\newtheorem{lem}[thm]{Lemma}
\newtheorem{prop}[thm]{Proposition}

\newtheorem{rem}[thm]{Remark}
\newtheorem{defn}[thm]{Definition}
\newtheorem{problem}{Problem}

\allowdisplaybreaks

\begin{document}

\title[Restricted invertibility of continuous matrix functions]{Restricted invertibility of continuous matrix functions}



\author[Fan et al.]{Adrian Fan}

\address{A. Fan, Department of Mathematics, Evans Hall, 3840, University of California, Berkeley, CA 94720-3840}

\email{adrianfan@berkeley.edu}


\author[]{Jack Montemurro}

\address{J. Montemurro, Department of Mathematics, University of Toronto, Toronto, M5S 2E4, Canada.}

\email{jack.montemurro@mail.utoronto.ca}


\author[]{Pavlos Motakis}

\address{P. Motakis, Department of Mathematics and Statistics, York University, 4700 Keele Street, Toronto, Ontario, M3J 1P3, Canada}

\email{pmotakis@yorku.ca}


\author[]{Naina Praveen}

\address{N. Praveen, Department of Mathematics, Ashoka University, Rajiv Gandhi Education City, P.O.Rai, Sonepat, Haryana 131 029, India.}

\email{naina.praveen@alumni.ashoka.edu.in}


\author[]{Alyssa Rusonik}

\address{A. Rusonik, Department of Mathematics, University of Toronto, Toronto, M5S 2E4, Canada.}

\email{alyssa.rusonik@mail.utoronto.ca}


\author[]{Paul Skoufranis}

\address{P. Skoufranis, Department of Mathematics and Statistics, York University, 4700 Keele Street, Toronto, Ontario, M3J 1P3, Canada}

\email{pskoufra@yorku.ca}


\author[]{Noam Tobin}

\address{N. Tobin, Department of Mathematics and Statistics, York University, 4700 Keele Street, Toronto, Ontario, M3J 1P3, Canada}

\email{nsnow1@my.yorku.ca}


\thanks{
This research was conducted during a Fields Undergraduate Summer Research Program in 2021 and the first, second, fourth, and fifth authors were supported by this program. The third author was supported by NSERC Grant RGPIN-2021-03639. The sixth author was supported by NSERC Grant RGPIN-2017-05711. The seventh author was supported by an Undergraduate Student Research Award (NSERC)}

\date{\today}


\begin{abstract}
Motivated by an influential result of Bourgain and Tzafriri, we consider continuous matrix functions $A:\mathbb{R}\to M_{n\times n}$ and lower $\ell_2$-norm bounds associated with their restriction to certain subspaces. We prove that for any such $A$ with unit-length columns, there exists a continuous choice of subspaces $t\mapsto U(t)\subset \mathbb{R}^n$ such that for $v\in U(t)$, $\|A(t)v\|\geq c\|v\|$ where $c$ is some universal constant. Furthermore, the $U(t)$ are chosen so that their dimension satisfies a lower bound with optimal asymptotic dependence on $n$ and $\sup_{t\in \mathbb{R}}\|A(t)\|.$ We provide two methods. The first relies on an orthogonality argument, while the second is probabilistic and combinatorial in nature. The latter does not yield the optimal bound for $\dim(U(t))$ but the $U(t)$ obtained in this way are guaranteed to have a canonical representation as joined-together spaces spanned by subsets of the unit vector basis.
\end{abstract}

\maketitle

\section{Introduction}
This paper concerns itself with continuous matrix functions $A:\mathbb{R}\to M_{n\times n}$ restricted to linear subspaces and certain lower $\ell_2$-norm bounds satisfied on these subspaces. It is motivated by specific results in the non-continuous (static) setting.  Inspired by problems in harmonic analysis and the geometry of Banach spaces, a seminal 1987 article by Bourgain and Tzafriri (\cite{bourgain:tzafriri:1987}) proved the following.

\begin{thm}[\cite{bourgain:tzafriri:1987}] \label{OGBT}
There exist universal constants $c_0,d_0>0$ such that for any $A\in M_{n \times n}(\mathbb{R})$ with $\|Ae_i\|=1$ for $1\leq i\leq n$ there exists $\sigma \subset \{1,\dotsc, n\}$ with $|\sigma|\geq d_0{n\|A\|^{-2}}$ such that for any set of scalars $\{a_j\}_{j \in \sigma}$,
\begin{equation}
\label{OGBTeq1}
\norm{\sum_{j \in \sigma} a_jAe_j}_2 \ge c_0 \left(\sum_{j \in \sigma} |a_j|^2\right)^{1/2}.
\end{equation} 
Equivalently,  for $U_\sigma=\langle e_j : j\in\sigma \rangle$ and for any $v\in U_\sigma$, 
$\|A v\|\geq c_0\|v\|.$
\end{thm}

For the remainder of the paper, let $c_0$ and $d_0$ be a pair of constants satisfying the conclusion of Theorem \ref{OGBT}. We point out that a dimensional estimate of order ${n\|A\|^{-2}}$ is optimal up to a constant (see Remark \ref{optimal dimension}). Intuitively, the Bourgain-Tzafriri result guarantees the existence of a ``large'' subspace of $\mathbb{R}^n $ on which $A$ does not shrink vectors ``excessively.'' The above result implies that $A$ is invertible when restricted to $U$ which was the original motivation for \cite{bourgain:tzafriri:1987}, and hence the term restricted invertibility. The $U$ from the above theorem is a particularly simple type of subspace namely one spanned by a subset of standard basis vectors.

Theorem \ref{OGBT} from \cite{bourgain:tzafriri:1987} and subsequent work of Bourgain and Tzafriri (\cite{bourgain:tzafriri:1989}, \cite{bourgain:tzafriri:1991}) are strongly related to the famous Kadison-Singer conjecture \cite{kadison:singer:1959}. This was a central problem in $C^*$-algebras that was restated by Anderson in \cite{anderson:1979} as a problem about matrices and it was proved by Casazza and Tremain in \cite{casazza:tremain:2009} that a certain statement that is related to Theorem \ref{OGBT} is equivalent to this conjecture. Within this context (and others), Theorem \ref{OGBT} has been studied, reproved, and generalized many times including results by Vershynin in \cite{vershynin:2001}, by Spielman and Srivastava in \cite{spielman:srivastava:2012} (who showed that for $0<\varepsilon<1$ one can choose $c_0 = \varepsilon^2$ and $d_0 = (1-\varepsilon)^2$), and by Naor and Youssef in \cite{naor:youssef:2017}. Using some of the techniques developed in \cite{spielman:srivastava:2012}, Marcus, Spielman, and Srivastava eventually solved the Kadison-Singer conjecture in \cite{marcus:spielan:srivastava:2015}.

The preservation of lower $\ell_2$-norm bounds on subspaces is also related to a problem from infinite-dimensional Banach space theory, namely the factorization property of bounded linear operators with large diagonal. This problem has its origins in Andrew's paper \cite{andrew:1979} and it has been further developed by Laustsen, Lechner, and M\"uller in \cite{laustsen:lechner:mueller:2015}, by Lechner in \cite{lechner:2018:factor-mixed}, by Lechner, M\"uller, Schlumprecht, and the third author in \cite{lechner:motakis:mueller:schlumprecht:2020} and \cite{lechner:motakis:mueller:schlumprecht:2021}, and others. In the finite-dimensional Euclidean setting, the problem can be stated as follows: Given $n\in\mathbb{N}$ and $\theta>0$, determine $C>0$ and $m\in\mathbb{N}$ such that every norm-one $n\times n$ matrix $A =(a_{i,j})$ with $\min|a_{i,i}|\geq \theta$ is a $C/\theta$-factor of the $m\times m$ identity matrix $I_m$. This means that there exist matrices $L$ and $R$ of appropriate dimension such that $\|L\|\|R\|\leq C/\theta$ and $LAR = I_m$. Theorem  \ref{OGBT} guarantees that one may take $C =  c_0$ and $m \geq d_0 n\theta^2$. The continuous version of this problem was investigated by Dai, Hore, Jiao, Lan, and the third author in \cite{dai:hore:jiao:lan:motakis:2020} where non-optimal estimates were given.

We turn out attention towards formulating a version of Theorem \ref{OGBT}  in the setting of a continuous matrix function $A:\mathbb{R}\to M_{n \times n} $. A point-wise application of Theorem \ref{OGBT} gives that there exists a choice of subspaces of constant dimension $U(t)\subset \mathbb{R}^n$, for all $t\in\mathbb{R}$ such that $v\in U(t)$, $\|A(t)v\|\geq c\|v\|$.  However, the collection $\{U(t)\}_{t\in\mathbb{R}}$ is not a priori known to satisfy any useful properties. The property of focus in this paper is continuity. We say that a collection $\{U(t)\}_{t\in \mathbb{R}}$ of subspaces varies continuously or is a continuous collection of subspaces if the matrix function $P:\mathbb{R}\to M_{n\times n}$, assigning to each $t\in \mathbb{R}$ the orthogonal projection $P(t)$ onto $U(t)$, is continuous (see Definition \ref{continuous spaces notions} and Theorem \ref{projection yields basis}). We focus on generalizing Theorem \ref{OGBT} by addressing the following main problem:

\begin{problem}\label{mainquestion}
Given a continuous matrix function $A: \mathbb R \to M_{n \times n} (\mathbb R)$ satisfying pointwise the same hypotheses as those in Theorem \ref{OGBT}, is it possible to find subspaces $\{U(t)\}_{t \in \mathbb{R}}$ that vary continuously  satisfying a similar lower $\ell_2$-norm bound? If so, are there obtainable bounds on $\dim(U(t))$? 
\end{problem}

We use techniques from linear algebra, analysis, and combinatorics, and employ a probabilistic construction to obtain the desired collection $\{U(t)\}_{t\in\mathbb{R}}$ with a lower bound of the dimension which is indeed optimal up to a constant.  Closely related to Problem  \ref{mainquestion},  is the following problem considered in \cite{dai:hore:jiao:lan:motakis:2020}.

\begin{problem}\label{equivalentquestion}
Given $n\in\mathbb{N}$ and $\theta >0$, determine $C>0$ and $m\in\mathbb{N}$ such that for every continuous matrix function $A = (a_{i,j}): \mathbb R \to M_{n  \times n}$ with $\|A(t)\| \le 1$ and $\min_i|a_{i,i}(t)|\geq \theta$ for every $t\in\mathbb{R}$ there exists continuous matrix functions $L: \mathbb R \to M_{m \times n}$, $R: \mathbb R \to M_{n \times m}$ so that $LAR = I_{m}$ and $\sup_t\|L(t)\|\|R(t)\| \leq C/\theta$.
\end{problem}

In Theorem \ref{C-factor_min-stretch} we will show that a solution to the first problem also offers one for the second.  We will provide two approaches to solving these problems.  The first one, presented in Section \ref{optimal}, allows us to recover the (up to universal constants) optimal bound:

\begin{thm}
\label{finalresultstatement}
Let $A: \mathbb{R} \rightarrow M_{n \times n}$ be a continuous matrix function, satisfying the property that for all $t \in \mathbb{R}$, $\|A(t)e_i\|=1$ for every  $ 1 \le i \le n$.  Let $\Lambda = \sup_t \|A(t)\|$ and $\gamma\in(0,1)$. Then, there exists a continuous family of $m$-dimensional subspaces $\{U(t)\}_{t \in \mathbb R}$ where $m \ge (d_0 n/7\Lambda^2)$ such that for every $t \in \mathbb{R},$ and every $v \in U(t)$,  $\|Av\| \geq \gamma c_0 \|v\|$.
\end{thm}
Note that the Hilbert-Schmidt norm of each $A(t)$ is $\sqrt{n}$ and therefore $\Lambda \leq \sqrt{n}$.

The second approach yields a more rigid collection of subspaces $\{U(t)\}_{t\in\mathbb{R}}$. As a trade-off, their dimensional estimate is weakened.

\begin{defn}
\label{convex combos of spaces}
An $m$-dimensional subspace $U$ of $\mathbb{R}^n$ is called a quadratic convex combination of disjoint basis vectors if there exist disjoint subsets $\sigma_1 = \{i_1<\cdots<i_m\}$, $\sigma_2 = \{j_1<\cdots<j_m\}$ of $\{1,\ldots,n\}$ and $\lambda\in[0,1]$ such that $U$ is spanned by the orthonormal sequence $u_k = \lambda^{1/2}e_{i_k}+(1-\lambda)^{1/2}e_{j_k}$, $k=1,\ldots,m$. If we wish to be more specific, we will say that $U$ is a $\lambda$-quadratic convex combination of $U_{\sigma_1}$ and $U_{\sigma_2}$.
\end{defn}

\begin{thm} \label{finalresultBTstatement}
There exists universal constants $c, d > 0$ such that for all continuous matrix functions $A: \mathbb{R} \rightarrow M_{n \times n}$ with the property that $\|A(t)e_i\|=1$ for all $t \in \mathbb{R}$ and $ 1 \le i \le n$, there exists a continuous family of $m$-dimensional subspaces $\{U(t)\}_{t \in \mathbb R}$ of $\mathbb{R}^n$ with $m \ge dn/\Lambda^4$ where $\Lambda = \sup_t \|A(t)\|$ such that $\|Av\| \geq c \|v\|$ for every $t \in \mathbb{R}$ and every $v \in U(t)$. Furthermore, each subspace $U(t)$ is a quadratic convex combination of disjoint basis vectors. 
\end{thm}

The impact of this improved choice of subspaces on the dimensional estimate is on the exponent of $\Lambda$. To motivate the second method, which is presented in Section \ref{BT-proof}, we begin with a consideration of the issues that arise with a pointwise application of the Bourgain-Tzafriri result. Were we to try to merely apply Theorem \ref{OGBT} pointwise given $A : \mathbb R \to M_{n \times n}$, we would be able to generate a suitable collection $\{\sigma_t\}_{t\in\mathbb{R}}$ of subsets of $\{1, \dotsc, n\}$. Of course, this is not enough: $U_{\sigma_t}$ cannot be chosen uniformly with respect to $t$, so there is no reason to suppose that for two arbitrarily close points $t_1, t_2$ the corresponding $U_{\sigma_{t_1}}, U_{\sigma_{t_2}}$ would serendipitously satisfy the required continuous transition property. To see this, it is enough to notice that satisfactory sets $\sigma_t$ at a point need not be unique.  

Instead we find a suitable countable sequence of points $(t_i)_{i \in \mathbb Z}$, apply a static result \`a la  Bourgain-Tzafriri at each such $t_i$ to first find a subset $\sigma_i$ of $\{1,\ldots,n\}$ and then conclude by means of a ``stitching'' argument. By this, we mean that we continuously pass between $U_{\sigma_{i}}$ and $U_{\sigma_{{i+1}}}$ on $[t_i, t_{i+1}]$ via quadratic convex combinations $U(t)$ of $U_{\sigma_i}$ and $U_{\sigma_{i+1}}$ while simultaneously preserving the desired lower $\ell_2$-norm bound of $A(t)$ on $U(t)$. Importantly, it is necessary that on the whole interval $[t_i,t_{i+1}]$, $A(t)$ already satisfies this bound on both $U_{\sigma_{i}}$ and $U_{\sigma_{{i+1}}}$. But this alone is not sufficient. Without further stipulations, $U(t)$ could fail to preserve the desired properties at some point (e.g., $U(t)$ and $\mathrm{ker}A(t)$ may momentarily intersect for some $t \in [t_i, t_{i+1}]$, which would immediately violate the lower $\ell_2$-norm bound). From this, it becomes clear that there must be some relationship between $\sigma_{i}$ and $\sigma_{{i+1}}$ that would preserve the minimal stretch property during the ``stitching'' procedure. It is sufficient to require that for all $t\in[t_i,t_{i+1}]$, $A(t)$ satisfies the desired bound on the whole subspace $U_{\sigma_i\cup\sigma_{i+1}}$.

To achieve this, we rely on an iterative application of a modified Bourgain-Tzafriri argument. The modification introduces a dependence of each subsequent column set on the one that precedes it; the iteration consists of passing to subsets of the generated column sets, and allows for the modification to be bilateral. Then, each column set depends on both the one that precedes it and on the one that will follow. This is where the loss of optimality occurs.

Because the iterative application of the modified Bourgain-Tzafriri result incurs the loss of optimal dimension, in the proof of Theorem \ref{finalresultstatement}, we base the argument on a different way of refining the column sets. We introduce a dependence that allows for a different ``stitching'' argument to be made which takes place at the level of the spaces $A(t)(U_{\sigma_i})$ and $A(t)(U_{\sigma_{i+1}})$. This is based on nice properties of orthogonal subspaces, to which one may conveniently pass with a judicious construction.
As the Bourgain-Tzafriri result is applied once and subsequent modifications come only at the cost of worse universal constants, this optimally solves Problem \ref{mainquestion} (and, by extension, Problem \ref{equivalentquestion}).

\section{Preliminaries}
In this section we recall various norms, recall the concept of continuous matrix functions, and discuss continuously varying subspaces of $\mathbb{R}^n$. We denote the standard basis for $\mathbb{R}^n$ by $\{e_1,\dotsc,e_n\}$ and assume the norm on $\mathbb{R}^n$ to be the 2-norm: $\norm{x} = (\sum_{i=1}^n x_i^2)^\frac{1}{2}$ for all $x \in \mathbb{R}^n$. We begin by defining useful quantities on $M_{m \times n}$, the space of all real valued $m\times n$ matrices.

\begin{defn}
\label{normdefs}
Let $A =(a_{i,j})\in M_{m \times n}$.
\begin{enumerate}[leftmargin=23pt,label=(\roman*)]

\item The operator norm of $A$ is defined as
\[\|A\| = \sup\Big\{ \|Ax\|:x\in\mathbb{R}^n, \|x\| = 1\Big\}.\]
If $U$ is a subspace of $\mathbb{R}^n$, let $\|A|_U\| = \sup\{\|Ax\|:x\in U,\|x\| = 1\}$.

\item The minimal stretch of $A$ is defined as

\[m_A = \inf\Big\{ \|Ax\|:x\in\mathbb{R}^n, \|x\| = 1\Big\}.\]
If $U$ is a subspace of $\mathbb{R}^n$, let $m_{A|_U} = \inf\{\|Ax\|:x\in U,\|x\| = 1\}$.

\item\label{HSdef} The Hilbert-Schmidt norm of $A$ is defined as

\[\|A\|_{\text{HS}} = \Big(\sum_{i=1}^n\|Ae_i\|^2\Big)^{1/2} = \Big(\sum_{i=1}^m\sum_{j=1}^n a_{ij}^2\Big)^{1/2}.\]

\end{enumerate}
\end{defn}

We point out that these quantities can be equivalently defined as follows. Let $(\sigma_i)_{i=1}^{m\wedge n}$ denote sequence of the singular values of $A$.  Then $\|A\| = \max\sigma_i$ whereas $m_A = \min\sigma_i$ and $\|A\|_{\text{HS}} = (\mathrm{tr}(A^TA))^{1/2} = (\sum_{i=1}^{m\wedge n}\sigma_i^2)^{1/2}$ (see, e.g., \cite[Theorem 7.4.3]{leon:1980}). These yield the following estimates.
\[
\|A\| \leq \|A\|_\mathrm{HS}\leq \big(\mathrm{rank}(A)\big)^{1/2}\|A\|.
\]

\begin{defn}
\label{cntsmatrixfuncdef}
Let $I$ be an interval of $\mathbb{R}$. A matrix function $A = (a_{i,j}) : I \to M_{m \times n}$ is said to be continuous if its entries $a_{i,j}:I\to M_{m\times n}$ are continuous.
\end{defn}

It is well known that $A$ is continuous if and only if it is continuous as a function from the metric space $(I,|\cdot|)$ to the normed linear space $(M_{m \times n}, \|\cdot\|)$ (see, e.g., \cite[Lemma 3.1]{dai:hore:jiao:lan:motakis:2020}).

\begin{defn}
\label{continuous spaces notions}
Let $I$ be an interval of $\mathbb{R}$ and let $\{U(t)\}_{t \in I}$ be a family of subspaces of $\mathbb{R}^n$.
\begin{enumerate}[leftmargin=19pt,label=(\roman*)]

\item\label{continuous spaces} The family $\{U(t)\}_{t \in I}$ is said to be a continuous choice of subspaces of $\mathbb{R}^n$ if the function $P : I \to M_{n\times n}$ where $P(t)$ is the orthogonal projection onto $U(t)$ for all $t \in I$ is continuous.

\item\label{continuous bases} The family $\{U(t)\}_{t \in I}$ is said to admit a continuous choice of basis if, for some $m\in\mathbb{N}$, there exist continuous $\gamma_1,\ldots,\gamma_m:I\to \mathbb{R}^n$ so that for each $t\in I$, $\gamma_1(t),\ldots,\gamma_m(t)$ form a basis of $U(t)$.

\end{enumerate}
\end{defn}

\begin{lem}
\label{continuous gram-schmidt}
If $I$ is an interval of $\mathbb{R}$ and a family $\{U(t)\}_{t \in I}$ of subspaces of $\mathbb{R}^n$ admits a continuous choice of basis $\gamma_1,\ldots,\gamma_m:I\to \mathbb{R}^n$ then it also admits a continuous choice of orthonormal basis $u_1,\ldots,u_m:I\to\mathbb{R}^n$. Therefore $\{U(t)\}_{t \in I}$ is a continuous choice of subspaces of $\mathbb{R}^n$.
\end{lem}

\begin{proof}
Define $u_1(t) = \|\gamma_1(t)\|^{-1}\gamma_1(t)$. By induction on $k=2,\ldots,m$ let $\tilde u_k(t) = \gamma_k(t) - \sum_{i=1}^{k-1}\langle\gamma_k(t),u_i(t) \rangle u_i(t)$ and $u_k(t) = \|\tilde u_k(t)\|^{-1}\tilde u_k(t)$. This yields a continuous choice of orthonormal basis $u_1,\ldots,u_m:I\to\mathbb{R}^n$. Finally, note that for each $t\in I$, putting $W(t) = [u_1(t)\cdots u_m(t)]$ we have that $P(t) = W(t)W^T(t)$ the orthogonal projection onto $U(t)$ and $P:I\to M_{n\times n}$ is continuous.
\end{proof}

Our next goal is to show that \ref{continuous spaces} and \ref{continuous bases} of Definition \ref{continuous spaces notions} are in fact equivalent. Note it is essential that the domain $I$ is a subset of $\mathbb{R}$. Indeed if we replace $I$ with the unit circle $S^1$ then this would no longer be true since we are working with subspaces of $\mathbb{R}^n$.  For example, for all $x = (\cos(\theta),\sin(\theta))$ in $S^1$ let
\[U(\cos(\theta),\sin(\theta)) = \langle(\cos(\theta/2),\sin(\theta/2))\rangle\subset \mathbb{R}^2.\]
Then $\{U(x)\}_{x\in S^1}$ is a continuous choice of subspaces of $\mathbb{R}^2$ for which a continuous choice of basis is impossible.

\begin{thm}
\label{projection yields basis}
Let $I$ be an interval of $\mathbb{R}$.  A family of subspaces $\{U(t)\}_{t \in I}$ of $\mathbb{R}^n$ is continuous if and only if it admits a continuous choice of orthonormal basis.
\end{thm}

The proof of the above requires some preparatory steps. For each of the following lemmata, let $I$ be an interval of $\mathbb R$ and $P : I \to M_{n \times n}$ be a continuous matrix function such that $P(t)$ is an orthogonal projection for all $t \in \mathbb R$.

\begin{lem}\label{7a}
The rank of $P(t)$ is constant for all $t \in I$.  
\end{lem}

\begin{proof} 
Since the rank of an orthogonal projection is equal to the trace of the orthogonal projection, and since $\tr(P(t)): I  \to \mathbb R$ is continuous when $P$ is continuous, the result follows.
\end{proof}

\begin{lem} \label{7b-lem1}
For all $t_0 \in I$ there is an $\varepsilon > 0$ and continuous functions $\gamma_1, \dotsc, \gamma_k : (t_0 - \varepsilon , t_0 + \varepsilon)\cap I \to \mathbb R^n$ such that for all $t \in (t_0 - \varepsilon, t_0 + \varepsilon)\cap I, \, \Im(P(t)) = \langle \gamma_1(t), \dotsc, \gamma_k(t)\rangle$.
\end{lem}

\begin{proof}
Let $t_0 \in I$. By assumption, $\rank(P(t_0)) =k$, so there exist $k$ linearly independent columns of $P(t_0)$, $\{Pe_{i_j}(t_0)\}_{j=1}^{k}$, which span $\Im(P(t_0))$. That is, $$\Im(P(t_0)) = \langle Pe_{i_1}(t_0) \dotsc Pe_{i_k}(t_0) \rangle.$$ Since $\rank(Pe_{i_1}(t_0) \dotsc Pe_{i_k}(t_0))=k$, we can find $k$ suitable rows of the matrix $[Pe_{i_1}(t) \dotsc Pe_{i_k}(t)]$ to obtain $B : \mathbb R \to M_{k \times k}$ such that $B(t_0)$ is invertible (by virtue of satisfying $\det B(t_0) \neq 0$). 
The continuity of the determinant function guarantees the existence of an $\varepsilon > 0$ such that for any $t \in (t_0 - \varepsilon, t_0 + \varepsilon)\cap I, \, \det B(t) \neq 0$. 
Thus, for any $t \in (t_0 - \varepsilon, t_0 + \varepsilon)\cap I, \rank(Pe_{i_1}(t) \dotsc Pe_{i_k}(t)) = k$ and so $$\Im(P(t)) = \langle Pe_{i_1}(t) \dotsc Pe_{i_k}(t) \rangle  \text{ as } \rank(P(t))=k$$ (where $B(t)$ is the minor of $P(t)$ defined using the same components as $B(t_0)$).
\end{proof}

\begin{lem}\label{7b-lem2}
Let $a < c < b < d$ and $P: (a,d) \to M_{n \times n}$ be continuous and let $k=\rank (P(t))$ for all $t \in \mathbb R$. Assume that $\Im(P(t))$ admits a continuous choice of basis $\gamma_1, \dotsc, \gamma_k : (a, b) \to \mathbb R^n$ on $(a,b)$ and another continuous choice of basis on $(c,d)$. Then, $\gamma_1, \dotsc, \gamma_k$ may be extended to a continuous choice of basis on $(a,d)$.
\end{lem}

\begin{proof}
By Lemma \ref{7b-lem1}, it is sufficient to show that we can ``stitch together'' continuous bases with intersecting domains. Let
\begin{align*}\alpha_1, \dotsc, \alpha_k &: (a,b) \to \mathbb R^n \text{ s.t. } \forall t \in (a,b), \text{ Im}(P(t)) = \langle \alpha_1(t), \dotsc, \alpha_k(t) \rangle \\ 
\beta_1, \dotsc, \beta_k &: (c,d) \to \mathbb R^n \text{ s.t. } \forall t \in (c,d), \text{ Im}(P(t)) = \langle \beta_1(t), \dotsc, \beta_k(t) \rangle
\end{align*}
for some $a<c<b<d$. Let $t_0 \in (c,b)$.  Then we can write the $\alpha_i (t_0)$ as linear combinations of $\beta_1(t_0), \dotsc, \beta_k(t_0)$ since $\langle \alpha_1(t_0), \dotsc, \alpha_k(t_0)\rangle=\langle \beta_1(t_0), \dotsc, \beta_k(t_0) \rangle$. In other words, for each $i \in \{1, \dotsc, k\}$, there are $(\lambda_{ij})_{j=1}^k \text{ s.t. } \alpha_i(t_0) = \displaystyle \sum_{j=1}^k \lambda_{ij} \beta_j (t_0)$.

Since the $\left(\alpha_i(t_0)\right)_{i=1}^k$ and $\left(\beta_i(t_0)\right)_{i=1}^k$ are bases of Im$(P(t_0))$, the $\lambda_{ij}$ uniquely determine an invertible matrix $F = (\lambda_{ij}) \in M_{k \times k} (\mathbb R)$, for which 
$$ \begin{bmatrix} \alpha_1(t_0)^T \\ \vdots \\ \alpha_k(t_0)^T \end{bmatrix} = F \begin{bmatrix} \beta(t_0)^T \\ \vdots \\ \beta_k(t_0)^T \end{bmatrix}.$$ 
Define $(\gamma_1, \dotsc, \gamma_k) : (a,d) \to M_{n \times k}(\mathbb R)$ by: 
$$ \left(\gamma_1(t), \dotsc, \gamma_k(t)\right) =
    \begin{cases}
        \left(\alpha_1(t), \dotsc, \alpha_k(t)\right) & t\leq t_0 \\
        \left(\beta_1(t), \dotsc, \beta_k(t)\right) F^T & t > t_0. \\
        \end{cases}$$
Then for every $t \in (a,d),$ Im$(P(t)) = \langle \gamma_1(t), \dotsc, \gamma_k(t) \rangle$ (since $\det F \neq 0$), and the $\gamma_i$ are continuous.
\end{proof}

\begin{proof}[Proof of Theorem \ref{projection yields basis}]
Let $P: I \to M_{n\times n}$ be continuous such that for all $t \in I$, $P(t)$ is an orthogonal projection. We want to construct $u_i :I \to \mathbb R^n$ continuous such that $u_1(t), \dotsc, u_k(t)$ form a basis of Im$(P(t))$ for every $t \in I$. For simplicity, let us assume that $I = \mathbb{R}$ as the other cases are similar.

We work first with $[-1, 1]$. As it is compact, by Lemma \ref{7b-lem1}, $[-1, 1]$ admits a finite cover by open intervals $(a_j, b_j)$ so that there exist continuous $\gamma_{ji} : (a_j, b_j) \to \mathbb R^n$ such that for $t \in (a_j, b_j), (\gamma_{j1}(t), \dotsc, \gamma_{jk}(t))$ is a basis of Im$(P(t))$.
    
Without loss of generality, we assume $a_{j+1} < b_j$ to obtain a cover by ``interlocking" intervals (with non-empty sequential intersections). A finite, step-by-step application of Lemma \ref{7b-lem2} allows for the construction of continuous $\gamma_i : [-1, 1] \to \mathbb R^n$ such that $(\gamma_1(t), \dotsc, \gamma_k(t))$ is a basis of Im$(P(t))$ for each $t \in [-1, 1]$.
    
Suppose $\gamma_i(t)$ is defined for all $t \in [-m, m]$. We will extend $\gamma_i$ to $[-(m+1), m+1]$. We similarly use Lemma \ref{7b-lem1} and the compactness of $[-(m+1), m+1]$ to obtain a suitable finite open cover, whereupon we apply Lemma \ref{7b-lem2} as above to define $\gamma_i(t)$ for $t \in [-(m+1), m+1]$ without modifying $\gamma_i$ on $[-m, m]$. Notice that $\gamma_i$ is well-defined as it is independent of the choice of $m$ (subsequent extensions of $\gamma_i$ do not alter the behaviour of the function on a domain on which it was previously defined). 
\end{proof}

\section{Method I: Optimal Dimensional Bound}\label{optimal}
In this section we prove Theorem \ref{finalresultstatement} and then show how it can be used to provide a solution to Problem \ref{equivalentquestion}.

\subsection{The proof of Theorem \ref{finalresultstatement}}
We start by proving some statements that will allow us to find appropriate subspaces of $\mathbb{R}^n$ and then stitch them together.

\begin{prop}\label{halfreduc}
If $X$ and $Y$ are linear subspaces of $\mathbb{R}^n$ and $\dim{X} = \dim{Y} = m$, then there exist linear subspaces $\stilde X \subset X, \tilde Y \subset Y$ such that $\dim{\tilde X}, \dim{\stilde Y} \geq \lfloor \frac{m}{2} \rfloor$ and $\tilde X \perp \tilde Y$.
\end{prop}

In order to show this theorem, we first prove a lemma that allows us to construct a convenient basis for our subspaces.

\begin{lem} \label{perpendicular}
Given two subspaces $X$ and $Y$ of $\mathbb{R}^n$ of dimension $m$, there exist orthonormal bases $\{ x_1,x_2, \ldots, x_m\}$ and $\{ y_1,y_2, \ldots, y_m\}$ of $X$ and $Y$ respectively, such that $\langle x_i, y_j \rangle =0$\ for any $i \neq j$.
\end{lem}
\begin{proof}
Define $P: \mathbb{R}^n \rightarrow Y$ to be the projection map onto $Y$. Let $\stilde P= P|_{X}$, and so, $\stilde P: X \rightarrow Y$. Define $A=\stilde P^* \stilde P: X \rightarrow X$. Then $A$ is self-adjoint, so the Spectral Theorem implies there is an orthonormal basis $\{x_1, \ldots, x_m\}$ of $X$ such that each $x_i$ is an eigenvector of $A$. Now, let $\stilde y_i=\stilde Px_i$ and $\sigma = \{1\leq i\leq m: \tilde y_i\neq 0\}$. For $i\in\sigma$ let $y_i=\|\stilde y_i\|^{-1}\stilde y_i$. Note for $i, j\in\sigma$ with $i \neq j$ that
$$\langle \tilde y_i, \tilde y_j \rangle= \langle \stilde Px_i, \stilde Px_j \rangle= \langle x_i, \stilde P^* \stilde P x_j \rangle = \langle x_i, Ax_j \rangle = \lambda_j \langle x_i, x_j \rangle =0$$
so the sequence $(y_i)_{i\in\sigma}$ is orthonormal. Extend $(y_i)_{i\in\sigma}$ to an orthonormal basis  $(y_i)_{i=1}^m$ of $Y$. It now remains to check that $\langle x_i, y_j \rangle =0$  for $i \neq j$. If $\stilde P x_i=0$, then $x_i$ is orthogonal to every vector in $Y$. Otherwise, notice that $u_i \in \mathbb{R}^n= Y \oplus Y^{\perp}$, and so we can express it as $x_i=\tilde y_i+y'_i$ where $y'_i = x_i-\tilde y_i \in Y^{\perp}$. Then, it follows:
\[\langle x_i, y_j \rangle = \langle \tilde y_i + y'_i , y_j \rangle = \langle \tilde y_i, y_j \rangle + \langle y'_i, y_j \rangle = 0.\qedhere\]
\end{proof}

We can now proceed with the proof of Proposition \ref{halfreduc}

\begin{proof}

Lemma \ref{perpendicular} yields orthonormal bases $\{x_1, \ldots , x_m\} $ and $\{y_1, \ldots, y_m\}$ of $X$ and $Y$ respectively, such that $\langle x_i, y_j \rangle =0$ when $i \neq j$. When $m$ is odd, define $\stilde X= \langle\{x_1, \ldots, x_{\lfloor \frac{m}{2}\rfloor}\}\rangle$ and $\stilde Y= \langle \{ y_{\lceil \frac{m}{2}\rceil}, \ldots, y_{m-1}, y_m \}\rangle$. Then $\langle x_i, y_j \rangle =0$ for all $1 \leq i \leq \lfloor \frac{m}{2} \rfloor, \lceil \frac{m}{2} \rceil \le j \le m$.

If $m$ is even, define $\stilde X= \langle\{x_1, \ldots, x_{\frac{m}{2}}\}\rangle$ and $\stilde Y= \langle \{ y_{\frac{m}{2}+1}, \ldots, y_{m-1}, y_m\}\rangle$. Then $\langle x_i, y_j \rangle =0$ for all $1 \leq i \leq  \frac{m}{2},  \frac{m}{2}+1 \le j \le m$. Thus, $\stilde X \perp \stilde Y$.
\end{proof}

By Proposition \ref{halfreduc}, given any two subspaces, we can always find orthogonal subspaces at the cost of reducing the dimension by half. The following result again appeals to Lemma \ref{perpendicular} in order to show that given two subspaces of a fixed dimension, it is possible to traverse from one subspace to another through a continuously varying choice of subspaces of the same dimension.

\begin{prop}\label{continuouschoice} 
Let $X, Y \subset \mathbb{R}^n$, with $\dim(X)=\dim(Y)=m$.  Then, for any $a<b\in\mathbb{R}$, there exists a continuous choice of subspaces $\{U(t)\}_{t \in [a,b]}$ of $\mathbb{R}^n$, such that $X=U(a)$ and $Y=U(b)$, and $U(t)$ lies in $X+Y$ for all $t \in [a,b]$.
\end{prop}

\begin{proof}
Using Lemma \ref{perpendicular}, we can find an orthonormal bases $\{x_1, \ldots x_m\}$ and $\{y_1, \ldots y_m\}$ of $X$ and $Y$ respectively such that $\langle x_i, y_j \rangle =0$ whenever $i \neq j$. Now define for $1 \le i \le m$,

$$u_{i}(t)=\left\{\begin{array}{ll}
\left(1-\dfrac{t-a}{b-a}\right) x_i + \left(\dfrac{t-a}{b-a}\right)y_i & \text{if $x_i$, $y_i$ are linearly independent} \\
y_i & \text{if $x_i$, $y_i$ are linearly dependent.}
\end{array}\right.
$$

Let $U(t)= \langle u_1(t), u_2(t), \ldots, u_m(t) \rangle$. Observe that $U(a)=X$ and $U(b)=Y$. Notice that for any $t \in [a,b]$, at no point do we have linear dependence between the $u_i$'s as, in fact, they are always orthogonal. Thus, $\{U(t)\}_{t \in [a,b]}$ as defined above is our required family of subspaces.
\end{proof}

The following lemma guarantees the existence of a discrete collection of points in $\mathbb{R}^n$ such that the application of Theorem \ref{OGBT} on all points of this collection preserves the desired minimal stretch property on overlapping intervals covering $\mathbb{R}$.

\begin{lem}
\label{compactification}
Let $\varepsilon>0$ and $A:\mathbb{R}\to M_{n\times n}(\mathbb{R})$ be a continuous matrix function. Then there exists a sequence $(t_i)_{i\in\mathbb{Z}}$ in $\mathbb{R}$ such that the following are satisfied.

\begin{enumerate}[leftmargin=19pt,label=(\roman*)]

\item\label{compactification item1} For all $i\in\mathbb{Z}$, $t_i < t_{i+1}$, $\sup_it_i = \infty$, and $\inf_it_i = -\infty$. 

\item\label{compactification item2} For all $i\in\mathbb{Z}$ and $t$ in $[t_{i-1},t_{i+1}]$, $\|A(t_i) - A(t)\| \leq \varepsilon$. 

\end{enumerate}
In particular, if $i\in\mathbb{Z}$, $c>0$, and $U$ is a subspace of $\mathbb{R}^n$ such that $m_{A(t_i)|_U}\geq c$ then for all $t\in[t_{i-1},t_{i+1}]$, $m_{A(t)|_U}\geq c - \varepsilon$.
\end{lem}

\begin{proof}
By the uniform continuity of $A$ on each compact interval $[i,i+1]$, there exists $k_i\in\mathbb{N}$ such that for all $s,t$ in $[i,i+1]$ with $|s - t| \leq 1/k_i$, $\|A(s) - A(t)\| \leq \varepsilon/2$. Denote $F_i = \{i+j/k_i: j=0,1,\ldots,k_i-1\}$ and define $F = \cup_{i\in\mathbb{Z}}F_i$. Clearly, this is a discrete subset of $\mathbb{R}$ that is unbounded above and unbounded below. Enumerating it in increasing order yields the desired sequence $(t_i)_{i\in\mathbb{Z}}$ that satisfies \ref{compactification item1} and \ref{compactification item2}.

For the final part let $i\in\mathbb{Z}$, $c>0$, and $U$ be a subspace of $\mathbb{R}^n$ such that $m_{A(t_i)|_U}\geq c$. Then, for all $x\in U$ with $\|x\| = 1$ and $t\in[t_{i-1},t_{i+1}]$,
\[\|A(t)x\| \geq \|A(t_i)x\| - \|A(t_i) - A(t)\|\|x\|\geq c-\varepsilon.\]
In other words, $m_{A(t)|_U}\geq c - \varepsilon$.
 \end{proof}

We can now prove Theorem \ref{finalresultstatement}, which we restate for convenience.

\begin{thm}
\label{finalresult}
Let $A: \mathbb{R} \rightarrow M_{n \times n}$ be a continuous matrix function, satisfying the property that for all $t \in \mathbb{R}$, $\|A(t)e_i\|=1$ for every  $ 1 \le i \le n$.  Let $\Lambda = \sup_t \|A(t)\|$ and $\gamma\in(0,1)$. Then, there exists a continuous family of $m$-dimensional subspaces $\{U(t)\}_{t \in \mathbb R}$ where $m \ge (d_0n)/(7\Lambda^2)$ such that for every $t \in \mathbb{R},$ and every $v \in U(t)$,  $\|Av\| \geq \gamma c_0 \|v\|$.
\end{thm}

\begin{proof}
To begin, note it is always possible to find a continuous choice of one-dimensional subspaces by taking $U(t) = \langle e_1\rangle$, for all $t\in\mathbb{R}$. Let $m_0 = \lceil d_0n/\Lambda^2 \rceil$. We will show that there is continuous choice of $\lfloor m_0/4\rfloor$-dimensional subspaces that satisfies the conclusion. Therefore, we can achieve $1\vee\lfloor m_0/4\rfloor$ which dominates $m_0/7$ for all possible values of $m_0$.

Apply Lemma \ref{compactification} to $A$ for $0<\varepsilon < (1-c)c_0$ to find an increasing sequence of points $(t_i)_{i \in \mathbb Z}$ satisfying the conclusion of that lemma. We will later specify that $\varepsilon$ needs in fact to be taken even smaller. For each $i\in\mathbb{Z}$ apply Theorem \ref{OGBT} to $A(t_i)$ to obtain a set $\sigma_i$ of a common cardinality $m_0 \geq d_0n/\Lambda^2$. For each $i\in\mathbb{Z}$, by letting $U_i = \langle e_{j}: j \in \sigma_{i}\rangle$, we have for all $t\in[t_{i-1},t_{i+1}]$ that $m_{A(t)|_{U_i}}\geq \gamma c_0$. 

Let $i \in \mathbb Z$ be given. We will outline the mechanism by which one produces the required subspaces on $[t_i, t_{i+1}]$, and then conclude by extending the construction to all of $\mathbb R$ by working on each interval and ``stitching'' together at the boundaries.   Given $i$ and $t\in[t_{i-1},t_{i+1}]$ let $V_{i,t} :=\langle A(t)e_{j} : j \in \sigma_{i}\rangle = A(t)(U_i)$.

A diagram illustrating the choice of subspaces on the interval $[t_i, t_{i+2}]$ is included below for the reader's convenience.

\vspace{0.2in}

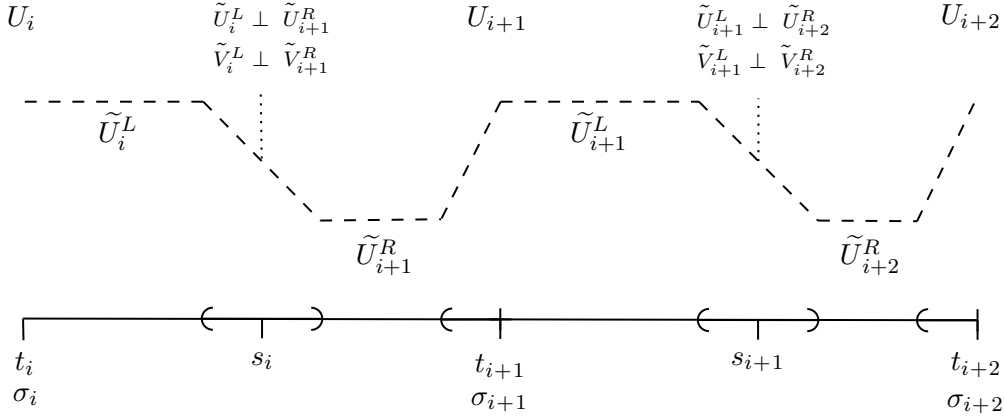
\tikzset{every picture/.style={line width=0.75pt}} 

\begin{tikzpicture}[x=0.75pt,y=0.75pt,yscale=-1,xscale=1]

\draw    (99.63,199.79) -- (581.01,200.65) ;
\draw    (99.63,199.79) -- (99.83,210.75) ;
\draw    (340.32,200.22) -- (340.32,210.08) ;
\draw    (220.15,199.89) -- (220.15,209.75) ;
\draw    (581.01,200.65) -- (581.01,210.51) ;
\draw    (470.32,200.56) -- (470.32,210.42) ;
\draw    (189.77,199.86) -- (250.53,199.92) ;
\draw [shift={(250.53,199.92)}, rotate = 180.06] [color={rgb, 255:red, 0; green, 0; blue, 0 }  ][line width=0.75]      (5.59,-5.59) .. controls (2.5,-5.59) and (0,-3.09) .. (0,0) .. controls (0,3.09) and (2.5,5.59) .. (5.59,5.59) ;
\draw [shift={(189.77,199.86)}, rotate = 0.06] [color={rgb, 255:red, 0; green, 0; blue, 0 }  ][line width=0.75]      (5.59,-5.59) .. controls (2.5,-5.59) and (0,-3.09) .. (0,0) .. controls (0,3.09) and (2.5,5.59) .. (5.59,5.59) ;
\draw    (439.94,200.52) -- (500.7,200.59) ;
\draw [shift={(500.7,200.59)}, rotate = 180.06] [color={rgb, 255:red, 0; green, 0; blue, 0 }  ][line width=0.75]      (5.59,-5.59) .. controls (2.5,-5.59) and (0,-3.09) .. (0,0) .. controls (0,3.09) and (2.5,5.59) .. (5.59,5.59) ;
\draw [shift={(439.94,200.52)}, rotate = 0.06] [color={rgb, 255:red, 0; green, 0; blue, 0 }  ][line width=0.75]      (5.59,-5.59) .. controls (2.5,-5.59) and (0,-3.09) .. (0,0) .. controls (0,3.09) and (2.5,5.59) .. (5.59,5.59) ;
\draw  [dash pattern={on 4.5pt off 4.5pt}]  (100.5,90.33) -- (190.5,90.33) ;
\draw  [dash pattern={on 4.5pt off 4.5pt}]  (190.5,90.33) -- (249.5,150.33) ;
\draw  [dash pattern={on 4.5pt off 4.5pt}]  (249.5,150.33) -- (310.5,149.67) ;
\draw    (310.5,200.33) -- (340.32,200.22) ;
\draw [shift={(340.32,200.22)}, rotate = 359.79] [color={rgb, 255:red, 0; green, 0; blue, 0 }  ][line width=0.75]    (-5.59,0) -- (5.59,0)(0,5.59) -- (0,-5.59)   ;
\draw [shift={(310.5,200.33)}, rotate = 359.79] [color={rgb, 255:red, 0; green, 0; blue, 0 }  ][line width=0.75]      (5.59,-5.59) .. controls (2.5,-5.59) and (0,-3.09) .. (0,0) .. controls (0,3.09) and (2.5,5.59) .. (5.59,5.59) ;
\draw    (550.5,200.67) -- (581.01,200.65) ;
\draw [shift={(581.01,200.65)}, rotate = 539.97] [color={rgb, 255:red, 0; green, 0; blue, 0 }  ][line width=0.75]    (0,5.59) -- (0,-5.59)   ;
\draw [shift={(550.5,200.67)}, rotate = 359.97] [color={rgb, 255:red, 0; green, 0; blue, 0 }  ][line width=0.75]      (5.59,-5.59) .. controls (2.5,-5.59) and (0,-3.09) .. (0,0) .. controls (0,3.09) and (2.5,5.59) .. (5.59,5.59) ;
\draw  [dash pattern={on 4.5pt off 4.5pt}]  (340.5,90.33) -- (310.5,149.67) ;
\draw  [dash pattern={on 4.5pt off 4.5pt}]  (340.5,90.33) -- (440.5,90.33) ;
\draw  [dash pattern={on 4.5pt off 4.5pt}]  (440.5,90.33) -- (500.5,150.33) ;
\draw  [dash pattern={on 4.5pt off 4.5pt}]  (500.5,150.33) -- (539.49,150.59) -- (550.5,150.67) ;
\draw  [dash pattern={on 4.5pt off 4.5pt}]  (550.5,150.67) -- (579.5,90.67) ;
\draw  [dash pattern={on 0.84pt off 2.51pt}]  (219.5,86.67) -- (220,120.33) ;
\draw  [dash pattern={on 0.84pt off 2.51pt}]  (470.5,120.33) -- (470.5,112.33) -- (470.5,88.67) ;

\draw (93,233.4) node [anchor=north west][inner sep=0.75pt]    {$\sigma _{i}$};
\draw (324,235.4) node [anchor=north west][inner sep=0.75pt]    {$\sigma _{i}{}_{+}{}_{1}$};
\draw (564,236.4) node [anchor=north west][inner sep=0.75pt]    {$\sigma _{i}{}_{+}{}_{2}$};
\draw (94,215.4) node [anchor=north west][inner sep=0.75pt]    {$t_{i}$};
\draw (326,216.4) node [anchor=north west][inner sep=0.75pt]    {$t_{i}{}_{+}{}_{1}$};
\draw (566,215.4) node [anchor=north west][inner sep=0.75pt]    {$t_{i}{}_{+}{}_{2}$};
\draw (213,215.4) node [anchor=north west][inner sep=0.75pt]    {$s_{i}$};
\draw (455,215.4) node [anchor=north west][inner sep=0.75pt]    {$s_{i}{}_{+}{}_{1}$};
\draw (136,96.4) node [anchor=north west][inner sep=0.75pt]    {$U_{i}^{L}$};
\draw (266,158.4) node [anchor=north west][inner sep=0.75pt]    {$U_{i+1}^{R}$};
\draw (374,96.4) node [anchor=north west][inner sep=0.75pt]    {$U_{i+1}^{L}$};
\draw (510,158.4) node [anchor=north west][inner sep=0.75pt]    {$U{_{i+2}^{R}}$};
\draw (139,94) node [anchor=north west][inner sep=0.75pt]    {$\widetilde{\ \ }$};
\draw (269,156) node [anchor=north west][inner sep=0.75pt]    {$\widetilde{\ \ }$};
\draw (377,94) node [anchor=north west][inner sep=0.75pt]    {$\widetilde{\ \ }$};
\draw (513,156) node [anchor=north west][inner sep=0.75pt]    {$\widetilde{\ \ }$};
\draw (90,40.4) node [anchor=north west][inner sep=0.75pt]    {$U_{i}$};
\draw (322,40.4) node [anchor=north west][inner sep=0.75pt]    {$U_{i}{}_{+}{}_{1}$};
\draw (561,39.4) node [anchor=north west][inner sep=0.75pt]    {$U_{i}{}_{+}{}_{2}$};
\draw (194,61.4) node [anchor=north west][inner sep=0.75pt]  [font=\scriptsize]  {$V_{i}^{L} \perp \ V_{i+1}^{R}$};
\draw (438,42.4) node [anchor=north west][inner sep=0.75pt]  [font=\scriptsize]  {$U_{i+1}^{L} \perp \ U_{i+2}^{R}$};
\draw (194,41.4) node [anchor=north west][inner sep=0.75pt]  [font=\scriptsize]  {$U_{i}^{L} \perp \ U_{i+1}^{R}$};
\draw (438,62.4) node [anchor=north west][inner sep=0.75pt]  [font=\scriptsize]  {$V_{i+1}^{L} \perp \ V_{i+2}^{R}$};
\draw (196,59) node [anchor=north west][inner sep=0.75pt]    {$\tilde{\ }$};
\draw (233,59) node [anchor=north west][inner sep=0.75pt]    {$\tilde{}$};
\draw (198,39) node [anchor=north west][inner sep=0.75pt]    {$\tilde{}$};
\draw (233,39) node [anchor=north west][inner sep=0.75pt]    {$\tilde{}$};
\draw (442,60) node [anchor=north west][inner sep=0.75pt]    {$\tilde{}$};
\draw (442,40) node [anchor=north west][inner sep=0.75pt]    {$\tilde{}$};
\draw (487,40) node [anchor=north west][inner sep=0.75pt]    {$\tilde{}$};
\draw (483,60) node [anchor=north west][inner sep=0.75pt]    {$\tilde{}$};

\end{tikzpicture}
\vspace{0.2in}

Fix an $s_i \in (t_i, t_{i+1})$. At $s_i$, we apply Proposition \ref{halfreduc} to the subspaces $V_{i, s_i}$ and $V_{i+1, s_i}$ to obtain subspaces $\stilde V^{L}_{i, s_i} \subset V_{i,s_i}$ and $\stilde V^{R}_{i+1, s_i} \subset V_{i+1, s_i}$ such that $\stilde V^{L}_{i, s_i} \perp \stilde V^{R}_{i+1, s_i}$ and $\dim \stilde V^L_{i, s_i}= \dim \stilde V_{i+1, s_i}^R= \lfloor \frac{m_0}{2} \rfloor$.

We once more invoke Proposition \ref{halfreduc} at the pre-images of $\stilde V_{i+1, s_i}^R$ and $\stilde V_{i, s_i}^L$ under $A$ to obtain subspaces $\tilde U^{L}_{i} \subset (A(s_i)|_{U_i})^{-1}\left(\stilde V_{i+1, s_i}^R\right) \subset U_{i}$ and $\stilde U^{R}_{i+1} \subset (A(s_i)|_{U_{i+1}})^{-1}\left(\stilde V_{i, s_i}^L\right)  \subset U_{i+1}$ such that $\stilde U^{L}_{i} \perp \stilde U^{R}_{i+1}$. The iterative application of Proposition \ref{halfreduc} incurs a further loss of dimension: for all $t \in [t_i, t_{i+1}],$ $$\dim \stilde U_{i, t}^L= \dim \stilde U_{i+1, t}^R =\left\lfloor\frac{1}{2}\left\lfloor\frac{m_0}{2}\right\rfloor\right\rfloor =  \left \lfloor \frac{m_0}{4} \right \rfloor.$$

Now, we must find a way to continuously ``stitch'' $\stilde U_i^L$ and $\stilde U_{i+1}^R$ around $s_i$, and similarly between $\stilde U_{i+1}^R$ and $\stilde U_{i+1}^L$ to pass to the next interval. This second ``stitch'' will be necessary as there is no reason to suppose that $\stilde U_{i+1}^R = \stilde U_{i+1}^L$. 

To this end, we pick some $\eta_i \in (0, \min \{\frac{s_i - t_i}{2}, \frac{t_{i+1} - s_i}{2}\})$. This will be the margin through at which the ``stitching'' occurs (that is, there will be two instances of ``stitching together'': on $(s_i - \eta_i, s_i + \eta_i)$ and on $(t_{i+1} - \eta_i, t_{i+1}]$). 

Observe that an application of Proposition \ref{continuouschoice} allows us to switch from $\stilde U^{L}_{i}$ at $s_i - \eta_i$ to $ \stilde U^{R}_{i+1}$ at $s_i+\eta_i$, through subspaces of dimension $m_0$ without violating the minimal stretch property as the collection of subspaces are contained in $\stilde U_{i}^L \oplus \stilde U_{i+1}^{R}$. This will be verified in Lemma \ref{minimalstretch}. 

It therefore remains only to ``stitch'' together on $(t_{i+1}-\eta_i, t_{i+1}] \subsetneq (s_i, t_{i+1}]$. Notice that $\stilde U_{i+1}^R \subset U_{i+1}$, and $\stilde U_{i+1}^L \subset U_{i+1}$. As the minimal stretch property holds for all vectors in $U_{i+1}$, the minimal stretch property also holds in $\stilde U_{i+1}^R+\stilde U_{i+1}^L$. Thus, we may appeal to Proposition \ref{continuouschoice} once again, and find a  collection of subspaces of dimension $\lfloor \frac{m_0}{4} \rfloor$ to traverse from $\stilde U_{i+1}^R$ at $t_{i+1}-\eta_i$, to $\stilde U_{i+1}^L$ at $t_{i+1}$.

To recapitulate, on each $[t_i, t_{i+1}]$, we take:
\begin{align*}
    \stilde U_{i}^L & \text{ for } t \in [t_i, s_i -\eta_i], \\
    \text{``stitch'' } \stilde U_{i}^L \text{ and } \stilde U_{i+1}^R & \text{ for } t \in (s_i -\eta_i, s_i + \eta_i), \\
    \stilde U_{i+1}^R & \text{ for } t \in [s_i + \eta_i, t_{i+1} - \eta_i], \\
    \text{``stitch'' } \stilde U_{i+1}^R \text{ and } \stilde U_{i+1}^L & \text{ for } t \in (t_{i+1} - \eta_i, t_{i+1}), \\
    \stilde U_{i+1, t}^L & \text{ for } t = t_{i+1}. \qedhere
\end{align*}
\end{proof}

\begin{lem} \label{minimalstretch}
For any vector $x$ lying in $\stilde U_i^L\oplus\stilde U_{i+1}^R, \ i \in \mathbb{Z}$, the vector $x$ satisfies $\|A(t)x\|\ge \gamma c_0\|x\|$, when $t \in [t_i, t_{i+1}]$.
\end{lem}

\begin{proof}
If $x \in \stilde U_{i}^L\oplus \stilde U_{i+1}^R$, then $x=f+g$, where $f \in \stilde U_i^L$ and $g \in U_{i+1}^R$. As $\stilde U_i^L$ and $\stilde U_{i+1}^R$ satisfy the minimal stretch property, we know that $\|A(t)f\| \ge c_0 \|f\|$ for $t \in [t_{i-1},t_{i+1}]$, and $\|A(t)g\| \ge c_0 \|g\|$ for $t \in [t_{i},t_{i+2}]$. Then for $t \in [t_{i-1},t_{i+1}]$, we have

\begin{align}
    \|A(t)x\|^2&=\|A(t)f\|^2+\|A(t)g\|^2+2\langle A(t)f, A(t)g\rangle \nonumber\\
    &\geq c_0^2\|f\|^2+ c_0^2\|g\|^2 + 2\langle A(t)f, A(t)g\rangle \nonumber\\
    &=c_0^2 \|x\|^2 +2\langle A(t)f, A(t)g\rangle. \label{eq7}
\end{align} 
Now recall that  $A(s_i)f \in \stilde V_{i+1, s_i}^R$ and $A(s_i)g \in \stilde V_{i, s_i}^L$ so
\begin{align*}
    |\langle A(t)f, A(t)g\rangle| &=|\langle A(t)f-A(s_i)f, A(t)g \rangle + \langle A(s_i)f, A(t)g -A(s_i)g \rangle+ \langle A(s_i)f, A(s_i)g \rangle|\\
    &\le \big(\big\|A(t)-A(s_i)\big\|\|f\|\big)\big(\big\|A(t)\big\|\|g\|\big)+ \big(\big\|A(s_i)\big\|\|f\|\big)\big(\big\|A(t)-A(s_i)\big\|\|g\|\big)\\
    &\phantom{\le} + 0\\
    &\leq \Lambda\|A(t)-A(s_i)\|\|x\|^2.
\end{align*}
From Lemma \ref{compactification} \ref{compactification item2}, we know that $\|A(t)-A(s_i)\| \le \varepsilon$. Substituting this in (\ref{eq7}), we obtain
\begin{align*}
    \|A(t)x\|^2 &\ge c_0^2\|x\|^2\Big(1-2\Lambda \varepsilon c_0^{-2}\Big).
\end{align*}
By taking $\varepsilon = c_0^2(1-\gamma^2)/(2\Lambda)$ in Lemma \ref{compactification}, we obtain $\|A(t)x\| \ge \gamma c_0\|x\|.$\end{proof}

\begin{rem}
\label{optimal dimension}
We sketch an argument that demonstrates that the dimensional estimate $n/\|A\|^2$ in Theorem \ref{OGBT} (and hence also in Theorem \ref{finalresultstatement}) is optimal, up to a constant. Assume that for $n\in\mathbb{N}$ and $1\leq \lambda \leq \sqrt{n}$, there is a $m(n,\lambda)\in\mathbb{N}$ such that for any $n\times n$ matrix $A$ with unit-length columns and $\|A\|\leq \lambda$, there exists $\sigma\subset\{1,\ldots,n\}$ with $|\sigma|\geq m(n,\lambda)$ and $m_{A|_{U_\sigma}} >0$. We will show that necessarily $m(n,\lambda) < 4 n/\lambda^2$. To find an $A$ that confirms this, take $m = \lceil n/\lambda^2 \rceil$ and using Euclidean division write $n = dm+r$. Find an orthonormal sequence $u_1,\ldots,u_m$ in $\mathbb{R}^n$ such that $\langle u_1,\ldots,u_m\rangle\perp \langle e_1,\ldots,e_r\rangle$. Define $U = [u_1\cdots u_m]\in M_{n\times m}$, $B = [U|\cdots|U]\in M_{n\times dm}$ (take $d$ copies of $U$), $D = [e_1\cdots e_r]\in M_{n\times r}$ and $A = [D|B]\in M_{n\times n}$. Then, $A$ has unit-length columns, $\|A\| = \sqrt{d}\leq\lambda$, and $\rank(A) = m+r < 2m\leq 4n/\lambda^2$. Therefore, for any $\sigma\subset\{1,\ldots,n\}$ such that $m_{A|_{U_\sigma}}>0$ must satisfy $m(n,\lambda)\leq |\sigma| \leq \rank(A) < 4n/\lambda^2$.
\end{rem}

\subsection{Continuous factorization of the identity}
In this subsection we explain how Theorem \ref{finalresultstatement} yields a solution to Problem \ref{equivalentquestion}. We begin with a lemma that will help us find the left matrix in the factorization of the identity.

\begin{lem}
\label{left} Let $I$ be an interval of $\mathbb{R}$ and $A:I\to M_{n \times m}$ be a continuous matrix function such that there exists $c>0$ with  $m_{A(t)} \geq c$ for all $t$. Then there exists a continuous matrix function $L:\mathbb{R}\to M_{m\times n}$ such that $L(t)A(t)=I_m$ and $\|L(t)\|\leq 1/c$ for all $t\in I$.
\end{lem}

\begin{proof}
For each fixed $t\in I$ we have that $m_{A(t)}>0$ and thus $A(t)$ has trivial kernel. This implies that the $m\times m$ matrix $A^T(t)A(t)$ has trivial kernel and is thus invertible. Since the matrix function $A^TA:I\to M_{m\times m}$ is continuous and pointwise invertible, $(A^TA)^{-1}:I\to M_{m\times m}$ is continuous as well (see, e.g., \cite[Lemma 3.2]{dai:hore:jiao:lan:motakis:2020}). Then, $L = (A^TA)^{-1}A^T:I\to M_{m\times n}$ is continuous and for each $t\in I$, $L(t)A(t) = I_m$. To find $\| L(t) \|$, for fixed $t\in I$, write the singular value decomposition
$A(t) = U \Sigma V^T$ where $U, V$ are unitary and $\Sigma$ is rectangular diagonal. A direct computation yields $L(t) = V\tilde \Sigma U^T$, where $\tilde \Sigma = (\Sigma^T\Sigma)^{-1}\Sigma^T$, which is the matrix formed by taking reciprocals of all non-zero diagonal elements of $\Sigma$ and then taking the transpose. Since $V, U$ are unitary and thus preserve the matrix norm under multiplication, $\|L(t)\| = \|\tilde \Sigma\| \leq 1/c$.
\end{proof}

To find the right matrix, we will use Theorem \ref{finalresultstatement} and Theorem \ref{projection yields basis}.

\begin{thm}
\label{C-factor_min-stretch}
Let $A:\mathbb{R}\to M_{n \times n}$ be a continuous matrix function and $\theta>0$ with $\|A(t)e_i\|\geq \theta$ for $1\leq i\leq n$ and $\|A(t)\| \leq 1$ for $t \in \mathbb{R}$. Then for any $\gamma\in(0,1)$ and $m\leq \lceil (d_0/7)n\theta^2\rceil$ there exist continuous matrix functions $L,R$ of appropriate dimensions such that $L(t)A(t)R(t)=I_m$ with $\|L(t)\| \|R(t)\|\leq C/\theta$ for all $t\in\mathbb{R}$, where $C = (\gamma c_0)^{-1}$.
\end{thm}

\begin{proof}
Let $a_k(t)$ denote the $k$th column of $A(t)$,
\[D(t) = \bigg[\frac{e_1}{\|a_1(t)\|}\cdots \frac{e_n}{\|a_n(t)\|}\bigg]\quad\text{ and }\tilde A(t) = A(t)D(t).\]
Note that
\[\|\tilde A(t)\| \leq \|A(t)\| \max_{1\leq i\leq n}\|a_i(t)\| \leq \frac{1}{\theta}.\]
Also, $A$ satisfies the hypothesis of Theorem \ref{finalresult}, so that we have the existence of a collection of $m$-dimensional subspaces $\{U(t)\}_{t\in \mathbb{R}}$ that vary continuously such that $m_{\tilde{A}(t)|_{U(t)}}\geq \gamma c_0$
for all $t\in\mathbb{R}$, where
\[m\geq \frac{d_0}{7} \frac{n}{\sup_t\|\tilde{A}(t)\|^2} \geq \frac{d_0}{7} n\theta^2.\]

By Theorem \ref{projection yields basis} there exists a continuous choice of orthonormal basis $u_1,\ldots,u_m:\mathbb{R}\to \mathbb{R}^n$. Define $W(t) = [u_1(t)\cdots u_m(t)]$, which is a continuous matrix function with the property $\Image(W(t)) = U(t)$ and for all $x\in\mathbb{R}^m$, $\|W(t)x\| = \|x\|$. Let $R(t) = D(t)W(t)$, which is continuous and satisfies $\|R(t)\| = \|D(t)\| \leq 1/\theta$.

Now since 
\begin{align*}
m_{A(t)R(t)} &= \inf\Big\{\|\tilde A(t)W(t)x\|: x\in\mathbb{R}^m,\|x\| = 1\Big\}\\
&= \inf\Big\{\|\tilde A(t)x\|:x\in U(t),\|x\| = 1\Big\} \geq \gamma c_0
\end{align*}
we can apply Lemma \ref{left} to $A(t)R(t)$ to show the existence of a continuous left inverse $L(t)$ of $A(t)R(t)$ that satisfies $\|L(t)\| \leq (\gamma c_0)^{-1}$ for all $t\in\mathbb{R}$.
\end{proof}

\section{Method II: Column Space Approach} \label{BT-proof}
In this section, we prove Theorem \ref{finalresultBTstatement}. The goal is to attain a continuous choice of subspaces on which a given matrix function satisfies the desired minimal stretch property that more closely resembles spaces spanned by a subset of the unit vector basis. These will be quadratic convex combinations of disjoint basis vectors (see Definition \ref{convex combos of spaces}). Therefore, we require a statement that guarantees the existence of suitable pairs of subspaces spanned by disjoint basis vectors that behave sufficiently well with one another so that they can be ``stitched'' together.

\subsection{A Bourgain-Tzafriri Theorem for disjoint subsets of the basis}
We first prove the static result that is necessary in the proof of Theorem \ref{finalresultBTstatement}.

\begin{thm}\label{BTbigthm}
There exist constants $1>d_1 > d_2 > d_3 > 0$ such that the following holds. For every $n\times n$ matrix $A$ with $\|A e_{i}\|=1$ for $i=1,\ldots,n$, and for every $\sigma_1\subset\{1,\ldots,n\}$ with $|\sigma_1|\leq d_1n\|A\|^{-2}$, there exists
\[\sigma_2\subset\{1,\ldots,n\}\setminus\sigma_1\text{ with }|\sigma_2|\geq d_2n\|A\|^{-2}\]
such that for any choice of scalars $\{a_{j}\}_{j \in {\sigma_1 \cup \sigma_{2}}}$, we have 
\[\Big\|\sum_{j \in \sigma_1 \cup \sigma_{2}} a_j Ae_j \Big\|\geq \frac{1}{16\sqrt{2}} \Big(\sum_{j \in \sigma_{2}} |a_j|^2\Big)^{1/2}.\]
If we additionally assume that $|\sigma_1| \geq d_2n\|A\|^{-2}$ and $d_1d_2n\|A\|^{-4}\geq 1$  then there also exist
\[\tau_1 \subset \sigma_1\text{ and }\stilde \sigma_{2} \subset \sigma_{2}\text{ with }|\tau_1|  \ge d_3 n\|A\|^{-4}\text{ and }|\stilde \sigma_{2}| \geq d_3 n\|A\|^{-4}\]
such that for any choice of scalars $\{a_{j}\}_{j \in \tau_1 \cup \stilde \sigma_{2}}$, we have 
\[\Big\|\sum_{j \in \tau_1 \cup \stilde \sigma_{2}} a_j Ae_j\Big\|\geq \frac{1}{32} \Big(\sum_{j \in \tau_1 \cup \stilde \sigma_{2}} |a_j|^2\Big)^{1/2}.\]
\end{thm}

In order to make the theorem more tractable, the proof is broken down into lemmata. The argument below follows the general shape of the argument first made in \cite{bourgain:tzafriri:1987}. In Lemmata \ref{BTstep1}, \ref{BTstep2}, and \ref{BTstep3}  below it is assumed that we are given an  $n\times n$ matrix $A$ that satisfies the assumptions of Theorem \ref{BTbigthm} (i.e., $\left\|A e_{i}\right\| = 1$ for $i=1,\ldots,n$).

\begin{rem}
\label{actual constants}
The constants in the Theorem \ref{BTbigthm} are $d_1 = 1/320$ (proof of Lemma \ref{BTstep1}), $d_2 = d_1/4$ (Lemma \ref{BTstep2}), and $d_3 = d_2^2/2$ (proof of Lemma \ref{BTstep3}).
\end{rem}

\begin{lem} \label{BTstep1}
There exists a constant $d_1>0$ such that the following holds: For every $\sigma_1\subset \{1,\ldots,n\}$ with $|\sigma_1| \leq d_1n\|A\|^{-2}$ there exists $\sigma_2\subset \{1,\ldots,n\}$ with $\sigma_2\geq d_1 n\|A\|^{-2}$ such that for every $i \in \sigma_2$,
\[\big\|P_{\langle Ae_j : j \in (\sigma_1 \cup \sigma_2) \setminus \{i\} \rangle} Ae_i\big\| < \frac{1}{\sqrt2}.\]
\end{lem}

\begin{proof}
Take $\delta=1/(8\|A\|^2)$ and $d_1 = 1/(8\cdot40)$. Let $\left\{\xi_{i}\right\}_{i \in \sigma_1^c}$ be a sequence of independent random variables of mean $\delta$ over a probability space $(\Omega, \Sigma, \mu)$ taking only the values 0 and 1.  Define
\[D = \Big\{\omega\in\Omega: \Big|\sum_{i\in\sigma_1^c}\xi_i(\omega) - \delta|\sigma_1^c|\Big|\geq \delta|\sigma_1^c|/2\Big\}.\]
By Bernstein's inequality (see, e.g., \cite[Lemma 1.3]{bourgain:tzafriri:1987} or \cite{bennet:1965}),
\begin{equation}
\label{Bernstein estimate}
\mu(D)\leq 2e^{-\delta|\sigma_1^c|/10}.
\end{equation}

For each $\omega \in \Omega$, let
$$
\sigma(\omega)=\left\{j \in \sigma_1^c : \xi_{j}(\omega)=1\right\} .
$$
We will show that there must be at least one $\sigma(\omega)$ which can be modified slightly to have the desired property. Set $V = \langle Ae_i : i \in \sigma_1 \rangle$. Then

\begin{align*}
    &\int_{\Omega} \sum_{i \in \sigma_1^c} \xi_{i}(\omega) \left\| P_{\langle \xi_{j} (\omega) Ae_j \cup V \rangle_{j \neq i, j \in \sigma_1^c}} (Ae_i) \right\|^2 d \mu \\
    &= \sum_{i \in \sigma_1^c} \bigg( \int_{\Omega} \xi_i(\omega) d\mu\bigg) \bigg( \int_{\Omega}\left\|P_{\langle \xi_j(\omega) Ae_j \cup V \rangle_{j \neq i, j \in \sigma_1^c}}(Ae_i)\right\|^2 d\mu\bigg)\\
    &=   \delta \int_{\Omega} \sum_{i \in \sigma_1^c} \left\| P_{\langle \xi_{j} (\omega) Ae_j \cup V \rangle_{j \neq i, j \in \sigma_1^c}} (Ae_i) \right\|^2 d \mu \\
    &\leq \delta \int_{\Omega} \sum_{ i \in \sigma_1^c}\left\|P_{\langle \xi_j(\omega) Ae_j \cup V \rangle_{j \in \sigma_1^c}} (Ae_i) \right\|^2 d\mu
\end{align*}
where the first equality used the fact that $ \xi_{i}(\omega)$ and $  \left\| P_{\langle \xi_{j} (\omega) Ae_j \cup V \rangle_{j \neq i, j \in \sigma_1^c}} (Ae_i) \right\|^2$ are independent (since the latter is a function of $(\xi_1,\ldots \xi_{i-1},\xi_{i+1},\ldots  \xi_n)$ so that we may split up the integral into a product) and the last line used the monotonicity of projection norms. By letting $W(\omega) = \{\xi_j (\omega) Ae_j : j \in \sigma_1^c\} \cup V,$ we have
\begin{align*}
&\int_{\Omega} \sum_{i \in \sigma_1^c} \xi_{i}(\omega) \left\| P_{\langle \xi_{j} (\omega) Ae_j \cup V \rangle_{j \neq i, j \in \sigma_1^c}} (Ae_i) \right\|^2 d \mu \\
  & \leq \delta \int_{\Omega} \sum_{ i \in \sigma_1^c}\left\|P_{\langle W(\omega) \rangle} (Ae_i) \right\|^2 d\mu \\
& \leq \delta \int_{\Omega} \sum_{i=1}^n \left\| P_{\langle W(\omega) \rangle} (Ae_i)\right\|^2 d \mu \\
&=\delta \int_{\Omega}\left\|P_{\langle W(\omega) \rangle} A\right\|_\mathrm{H S}^{2} d \mu \nonumber \\
& \leq \delta\|A\|^{2} \int_{\Omega} \left(|\sigma_1|+ \sum_{j \in \sigma_1^c} \xi_{j}(\omega)  \right) \ d \mu \\
&=\delta \|A\|^{2} \left(|\sigma_1|+ \delta |\sigma_1^c|  \right)
\end{align*}
where we used the standard inequality $\|A\|^2_\mathrm{H S}\leq \|A\|^2 \mathrm{rank}(A)$ and then made use of the fact that $\mathrm{rank}(P_{\langle W(\omega) \rangle} A)$ is bounded above by the number of non-zero vectors in $\{\xi_i(\omega): i\in \sigma_1^c\}\cup \{Ae_i: i\in \sigma_1\}$ which has the crude bound $|\sigma_1|+\sum_{j\in \sigma_1^c} \xi_j(\omega)$.

Since all functions involved are non-negative this yields:
\begin{align*}
&\int_{\Omega \setminus D} \, \sum_{i \in \sigma_1^c} \xi_{i}(\omega)\left\|P_{\langle \xi_{j}(\omega) Ae_{j} \cup V \rangle_{j \neq i, j \in \sigma_1^c}}\left(Ae_{i}\right)\right\|^{2} d \mu \\
&\leq \int_{\Omega} \, \sum_{i \in \sigma_1^c} \xi_{i}(\omega)\left\|P_{\langle \xi_{j}(\omega) Ae_{j} \cup V \rangle_{j \neq i, j \in \sigma_1^c}}\left(Ae_{i}\right)\right\|^{2} d \mu \\
&\leq \delta \|A\|^{2} \left(|\sigma_1|+ \delta |\sigma_1^c|  \right)
\end{align*}
which, with  \eqref{Bernstein estimate}, implies that there exists a point $\omega_{0} \in \Omega \setminus D$ 
such that 
\begin{align*}
    \displaystyle \sum_{i \in \sigma(\omega_0)} \left\|P_{\langle Ae_j  \cup V \rangle_{j \in \sigma(\omega_0) \setminus \{i\}}} (Ae_i) \right\|^2 
    &=\sum_{i \in \sigma_1^c} \xi_{i}(\omega_0)\left\|P_{\langle \xi_{j}(\omega_0) Ae_{j} \cup V \rangle_{j \neq i, j \in \sigma_1^c}}\left(Ae_{i}\right)\right\|^{2}\\
    &\le \frac{\delta \|A\|^2 (|\sigma_1|+ \delta |\sigma_1^c| )}{1 - \mu(D)}\\  
    &\leq \frac{\delta \|A\|^2 (|\sigma_1|+ \delta |\sigma_1^c| )}{1-2 \exp \bigg(\tfrac{-\delta|\sigma_1^c|}{10}\bigg)}.
\end{align*}
The exponential factor in the denominator can be dealt with by working with working with two separate cases: 

\noindent\textbf{Case 1: $\delta n< 40$}. In this case since $d_1 = 1/(8\cdot40)$ we have that

\begin{align*}
d_1 \frac{n}{\|A\|^2}&= \frac{n \delta}{40} < 1
\end{align*}
so that $|\sigma_1|=0,$ and for any $\sigma_2 \subset \{1,\dots n\}$,  with $|\sigma_2|=1$ the statement is vacuously true.

\noindent\textbf{Case 2:} $\delta n \geq 40$. Note that the assumption made on the columns of $A$ imply that $\|A\|^2 \geq 1 $.  So since $d_1 < 1/2$ and $|\sigma_1|\leq d_1 n\|A\|^{-2}$ we must have that $|\sigma_1^c|\geq n/2$. Some simple calculus shows that $1-2e^{-x}\geq 2/3$ for all $x\geq 2$ and thus \begin{align*}
  \displaystyle \sum_{i \in \sigma(\omega_0)} \left\|P_{\langle Ae_j  \cup V \rangle_{j \in \sigma(\omega_0) \setminus \{i\}}} (Ae_i) \right\|^2 \leq \frac{\delta \|A\|^2 (|\sigma_1|+ \delta |\sigma_1^c| )}{1-2 \exp \bigg(\tfrac{-\delta|\sigma_1^c|}{10}\bigg)} \leq \frac{3}{2}\delta \|A\|^2 (|\sigma_1|+ \delta |\sigma_1^c| ).
\end{align*}  By definition of $D$, we also have:
\begin{align*}
   \left|\sigma(\omega_{0})\right|=\sum_{i \in \sigma_1^c} \xi_{i}\left(\omega_{0}\right) \geq \delta |\sigma_1^c|/ 2. 
\end{align*}
Let
 $$\sigma_2 = \bigg\{ i \in \sigma(\omega_0) : \left\|P_{\langle Ae_j  \cup V \rangle_{j \in \sigma(\omega_0) \setminus \{i\}}} (Ae_i) \right\| < 2 \|A\| \sqrt{\delta} \bigg\}.$$
Note that, given our choice of $\delta,$ $\sigma_2$ satisfies the desired conclusion. We just need to provide a lower bound on its size. Now,
\begin{gather*}
\sum_{i \in \sigma(\omega_0) \setminus \sigma_2} \left\|P_{\langle Ae_j  \cup V \rangle_{j \in \sigma(\omega_0) \setminus \{i\}}} (Ae_i) \right\|^2 \leq    \sum_{i \in \sigma(\omega_0)} \left\|P_{\langle Ae_j  \cup V \rangle_{j \in \sigma(\omega_0) \setminus \{i\}}} (Ae_i) \right\|^2\\
\leq \frac{3}{2}\delta \|A\|^{2} \left(|\sigma_1|+ \delta |\sigma_1^c|  \right).
\end{gather*}
so, using the definition of $\sigma_2$, 
\[
4\|A\|^2 \delta | \sigma(\omega_0) \setminus \sigma_2| \leq \frac{3}{2}\delta \|A\|^2(|\sigma_1|+ \delta |\sigma_1^c| ).
\]
Since $|\sigma_1| \leq d_1 n\|A\|^{-2}$ and $|\sigma(\omega_{0})| \geq \delta |\sigma_1^c|/ 2 $, we have
\begin{align*}
    4\|A\|^2 \delta | \sigma(\omega_0) \setminus \sigma_2| &\leq \frac{3}{2}\delta \|A\|^2\left(d_1\frac{n}{\|A\|^2} + 2|\sigma(\omega_0)|\right).
\end{align*}
Solving this inequality for $|\sigma_2|$ yields
\begin{align*}
|\sigma_2| &\geq \frac{1}{4}|\sigma(\omega_0)| - \frac{3}{8} d_1\frac{n}{\|A\|^2} \\
&\geq \delta\frac{|\sigma_1^c|}{8} - d_1\frac{n}{\|A\|^2}\\
&= \frac{|\sigma_1^c|}{64\|A\|^2}- \frac{n}{320\|A\|^2} \\
&\geq \frac{n}{128\|A\|^2}- \frac{n}{320\|A\|^2} \geq \frac{n}{320\|A\|^2} = d_1\frac{n}{\|A\|^2}.\qedhere
\end{align*}
\end{proof}

\begin{lem}\label{BTstep2}
For every $\sigma_1\subset \{1,\ldots,n\}$ with $|\sigma_1| \leq d_1n\|A\|^{-2}$ there exists $\sigma_2\subset \{1,\ldots,n\}\setminus\sigma_1$ with $\sigma_2\geq (d_1/2) n\|A\|^{-2}$ such that for every $\{a_j\}_{j \in \sigma_1 \cup \sigma_{2}}$:
\[\norm{\sum_{j \in \sigma_1 \cup \sigma_{2}} a_jAe_j}\geq \frac{1}{4\sqrt{2|\sigma_{2}}|} \sum_{j \in \sigma_{2}} |a_j|.\]
\end{lem}

\begin{proof}
From Lemma \ref{BTstep1} there exists $\sigma \subset \{1, \dotsc, n\} \setminus \sigma_1$ with $|\sigma| \ge d_1 n/\|A\|^2$, such that if for every $i \in \sigma$ we let $Ae_i = x_i$, then 
$$\norm{P_{\langle Te_k : k \in (\sigma_1 \cup \sigma) \setminus \{i\} \rangle} x_i} < \frac{1}{\sqrt2}.$$
For every $i \in \sigma$, let $u_i' = x_i - P_{\langle Te_k : k \in (\sigma_1 \cup \sigma) \setminus \{i\} \rangle} x_i$.
Then, by orthogonality, $$\norm{u_i'}^2 = \norm{x_i}^2 - \norm{ P_{\langle Te_k : k \in (\tau \cup \sigma) \setminus \{i\} \rangle} x_i }^2
> 1 - \frac{1}{2} = \frac{1}{2}$$
and $\norm{u_i'} \leq 1$.
In addition, for $i \in \sigma_1 \cup \sigma$ and  $j \in \sigma$ with $i \neq j$,
\begin{align*}
\langle x_i, u'_j \rangle 
&= \langle x_i, x_j - P_{\langle Te_k : k \in (\sigma_1 \cup \sigma) \setminus \{j\} \rangle} x_j \rangle \\
&= \langle x_i, x_j \rangle - \langle x_i, P_{\langle Te_k : k \in (\sigma_1 \cup \sigma) \setminus \{j\} \rangle} x_j \rangle \\
&= \langle x_i, x_j \rangle - \langle P_{\langle Te_k : k \in (\sigma_1 \cup \sigma) \setminus \{j\} \rangle} x_i, x_j \rangle = 0.
\end{align*}
Similarly, for $i\in\sigma$
\begin{align*}
\langle x_i, u'_i \rangle &= \|x_i\|^2  - \langle x_i, P_{\langle Te_k : k \in (\sigma_1 \cup \sigma) \setminus \{i\} \rangle} x_i\rangle\\
& = 1 - \|P_{\langle Te_k : k \in (\sigma_1 \cup \sigma) \setminus \{i\} \rangle} x_i\|^2 > 1/2.
\end{align*}
For $i\in\sigma$ let $u_i = \|u_i'\|^{-1}u_i'$, so that $1\geq \langle x_i,u_i\rangle > 1/2$.  

If $E(X)$ denotes the expected value of the random variable $X$ over $\{-1,1\}^\sigma$ with the uniform probability measure, a simple calculation yields
\[\mathbb E\Big(\Big\|\sum_{i \in \sigma} \varepsilon_i u_i\Big\|^2 \Big)  = \sum_{i \in \sigma} \norm{u_i}^2 = |\sigma|.\]
Thus if
\[\mathcal{E} = \Big\{(\varepsilon_i)_{i \in \sigma} \in \{-1, 1 \}^{\sigma} : \Big\|\sum_{i \in \sigma} \varepsilon_i u_i\Big\| \le 2 \sqrt{|\sigma|} \Big\},\]
it follows then by Markov's inequality that
\[|\mathcal{E}| \ge \frac{3}{4} 2^{|\sigma|}.\]
By a theorem of Sauer and Shelah (see, e.g., \cite[Page 144]{bourgain:tzafriri:1987} or \cite{sauer:1972}), whenever $k$ satisfies
\begin{equation}
\label{sauer shelah assumption}
|\mathcal{E}| > \sum_{i=0}^{k-1} \binom{|\sigma|}{i},
\end{equation}
then there exists a subset $\sigma_2 \subset \sigma$ of cardinality $k$
such that for each tuple $(\varepsilon_i)_{i \in \sigma_2}$
there exists an extension $(\varepsilon_i)_{i \in \sigma}$ which belongs to $\mathcal{E}$.
Note that \eqref{sauer shelah assumption} holds for $k \geq \frac{|\sigma|}{2}$ and therefore we may choose $\sigma_2$ with $|\sigma_2| \geq (d_1n)/(2\|A\|^2)$ and $|\sigma_2| \geq |\sigma|/2$.

To see that $\sigma_2$ satisfies the conclusion, let $\{a_j\}_{j \in \sigma_1 \cup \sigma_{2}}$ be given. Define $\theta_i \in \{-1 , 1 \}$ for each $i \in \sigma_2$ so that $\theta_i a_i = |a_i|$.
Then let $(\varepsilon_i)_{i \in \sigma}$ be an extension of $(\theta_i)_{i \in \sigma_2}$ that belongs to $\mathcal{E}$. It follows that
\begin{align*}
\Big|\Big\langle \sum_{j \in \sigma} \varepsilon_j u_j, \sum_{j \in \sigma_1 \cup \sigma_2} a_j x_j \Big \rangle \Big| 
&\le \Big\|\sum_{j \in \sigma} \varepsilon_j u_j\Big\|\Big\|\sum_{j \in \sigma_1 \cup \sigma_2} a_j x_j\Big\| \\
& \le 2 \sqrt{|\sigma|} \Big\|\sum_{j \in \sigma_1 \cup \sigma_2} a_j x_j\Big\| \\
&\le 2\sqrt{2|\sigma_2|}\Big\|\sum_{j \in \sigma_1 \cup \sigma_2} a_j x_j\Big\|.
\end{align*}
Hence 
\begin{align*}
2 \sqrt{2} \sqrt{|\sigma_2|}  \Big\|\sum_{j \in \sigma_1 \cup \sigma_2} a_j x_j\Big\|
&\ge \Big|\Big\langle \sum_{j \in \sigma} \varepsilon_j u_j, \sum_{j \in \sigma_1 \cup \sigma_2} a_j x_j\Big \rangle \Big| = \sum_{j \in \sigma_2} |a_j| \langle x_j, u_j \rangle \ge  \frac{1}{2}\sum_{j \in \sigma_2} |a_j|. \qedhere
\end{align*}
\end{proof}

In the next step, we will require Khintchine's inequality which states the following. Let $\{1,-1\}^n$ be endowed with the uniform probability measure. For $\{b_\ell\}_{\ell=1}^n$ in $\mathbb{R}$ we have
\[\mathbb{E}\left(\left|\sum_{\ell=1}^{n} \varepsilon_{\ell} b_{\ell}\right|\right) \ge \frac{1}{\sqrt{2}}\left(\sum_{\ell=1}^{n} \left|b_{\ell}\right|^2\right)^{1/2}.\]

\begin{lem}\label{BTstep3}
For every $\sigma_1\subset \{1,\ldots,n\}$ with $|\sigma_1| \leq d_1n\|A\|^{-2}$, there exists $\sigma_2\subset \{1,\ldots,n\}\setminus\sigma_1$ with $\sigma_2\geq (d_1/4) n\|A\|^{-2}$ such that for every $\{a_j\}_{j \in \sigma_1 \cup \sigma_{2}}$ in $\mathbb{R}$:
\[\norm{\sum_{j \in \sigma_1 \cup \sigma_{2}}a_jAe_j} \ge \frac{1}{16\sqrt{2}} \left(\sum_{j \in \sigma_{2}}|a_j|^2\right)^{1/2}.\]
\end{lem}

\begin{proof}
Consider the set $\sigma$ supplied by applying Lemma \ref{BTstep2} and denote $c' =1/(4\sqrt{2})$. We need to establish a subset $\sigma_2 \subset \sigma$ of cardinality $|\sigma_2| \geq |\sigma|/2$ such that for any choice of coefficients in $\{a_j\}_{j \in \sigma_1 \cup \sigma_{2}}$:
\[\norm{\sum_{j \in \sigma_1 \cup \sigma_{2}}a_jAe_j} \ge \frac{c'}{4} \left(\sum_{j \in \sigma_{2}}|a_j|^2\right)^{1/2}.\]

Suppose, for contradiction, that such a subset does not exist. Let $Ae_j=x_j$. Put $\upsilon_{1}=\sigma$. Then there exists a vector $y_1=\sum_{k \in \upsilon_1 \cup \sigma_1}b_{1,k}x_k$ such that $\|y_1\| < c'/4$, but $\sum_{k \in \upsilon_1} \left| b_{1,k} \right|^2 =1$.

Assume that we have already constructed subsets $\upsilon_1 \supset \upsilon_2 \supset \cdots \supset \upsilon_{p}$ with $|\upsilon_p| \geq |\sigma|/2$, and vectors $\{y_{\ell}\}_{\ell=1}^p$ such that $y_{\ell}=\sum_{k \in \upsilon_{\ell} \cup\sigma_1}b_{\ell,k}x_k$ and $\|y_{\ell}\| < c'/4$ and $\sum_{k \in \upsilon_{\ell}} \left| b_{{\ell},k} \right|^2 =1$, for $1 \leq \ell \leq p$. Consider the set
$$\upsilon_{p+1}=\left\{k \in \upsilon_p \ : \ \sum_{\ell=1}^{p}\left| b_{\ell,k} \right|^2 <1 \right\}.$$ 
If $|\upsilon_{p+1}| < |\sigma|/2$, then stop the procedure. On the other hand, if $|\upsilon_{p+1}| \geq |\sigma|/2$, then there exists a vector 
$$y_{p+1}= \sum_{k \in \upsilon_{p+1} \cup \sigma_1} b_{p+1, k} x_k$$
such that $\|y_{l+1}\|< c'/4$ and $\sum_{k \in \upsilon_{p+1}} \left| b_{j,k} \right|^2 =1$.

By the pigeon-hole principle, this algorithm must eventually terminate, say after $m$ steps. Then

$$|\upsilon_{m+1}| < \frac{|\sigma|}{2}$$
and thus, for $\ell \in \sigma \setminus \upsilon_{m+1}$, we have that
$$\sum_{\ell=1}^{m} \left| b_{\ell,k} \right| ^2 \geq 1$$
with the convention that $b_{\ell,k}=0$ for those $\ell$ and $k$ for which $b_{\ell, k}$ has not been defined.

Hence
\begin{align*}
    m&=\sum_{\ell=1}^{m} \sum_{k \in \upsilon_{\ell}}\left| b_{\ell,k} \right|^2 = \sum_{k \in \sigma} \sum_{\ell=1}^{m} \left| b_{\ell,k} \right|^2 \geq \sum_{k \in \sigma \setminus \upsilon_{m+1}} \sum_{\ell=1}^{m} \left| b_{\ell,k} \right| ^2 \geq |\sigma \setminus \upsilon_{m+1}|
\end{align*}
and this implies that $m \geq |\sigma| / 2$.

On the other hand, we have 
\begin{align}
\frac{c' \sqrt{m}}{4} >\left(\sum_{\ell=1}^{m}\left\|y_{\ell}\right\|^{2}\right)^{\frac{1}{2}} &= \left(\int\left\|\sum_{\ell=1}^{m} \varepsilon_{\ell} y_{\ell}\right\|^2 d \varepsilon\right)^{\frac{1}{2}}\nonumber \\
& \geq \int\left\|\sum_{\ell=1}^{m} \varepsilon_{\ell} y_{\ell}\right\| d \varepsilon \nonumber \\
&=\int\left\|\sum_{k \in \sigma \cup \sigma_1}\left(\sum_{\ell=1}^{m} \varepsilon_{\ell} b_{\ell, k}\right) x_{k}\right\| d\varepsilon \nonumber \\
& \geq \frac{c'}{\sqrt{\left|\sigma\right|}} \sum_{k \in \sigma} \int\left|\sum_{\ell=1}^{m} \varepsilon_{\ell} b_{\ell, k}\right| d \varepsilon  \label{eq1}\\
& \geq \frac{c'}{\sqrt{2\left|\sigma\right|}} \sum_{k \in \sigma}\left(\sum_{\ell=1}^{m}\left|b_{\ell, k}\right|^{2}\right)^{\frac{1}{2}} \label{eq2}
\end{align}
where (\ref{eq1}) follows from Lemma \ref{BTstep2}, and (\ref{eq2}) follows from Khintchine's inequality.

However, the inductive construction implies 
\begin{align}
\sum_{\ell=1}^{m}\left|b_{\ell, k}\right|^{2} \leq 2 \label{eq3}
\end{align}
for all $k \in \sigma$. This is clear if $k \in \upsilon_{m+1}$ because of how we construct $\upsilon_{m+1}$, while if $k \in \left(\upsilon_{p} \setminus \upsilon_{p+1}\right)$ for some $1 \leq p \leq m$, we have that
$$\sum_{\ell=1}^{m}\left|b_{\ell, k}\right|^{2}=\sum_{\ell=1}^{p-1}\left|b_{\ell, k}\right|^{2}+\left|b_{p, k}\right|^{2}<2.$$

It follows from (\ref{eq2}) and (\ref{eq3}) that

\begin{align*}
\cfrac{\sqrt{m\left|\sigma\right|}}{2} &> \sqrt{2} \sum_{k \in \sigma}  \left(\sum_{\ell=1}^{m}\left|b_{\ell, k}\right|^{2}\right)^{\frac{1}{2}}\\
&\geq \sum_{k \in \sigma}\left(\sum_{\ell=1}^{m}\left|b_{\ell, k}\right|^{2}\right)^{\frac{1}{2}}  \left(\sum_{\ell=1}^{m}\left|b_{\ell, k}\right|^{2}\right)^{\frac{1}{2}}\\
&= \sum_{k \in \sigma} \sum_{\ell=1}^{m}\left|b_{\ell, k}\right|^{2}=m.
\end{align*}
Thus $m < |\sigma|/4$ which contradicts the fact that $m \geq \left|\sigma\right| / 2$.
\end{proof}

We now prove Theorem \ref{BTbigthm}.

\begin{proof}[Proof of Theorem \ref{BTbigthm}]
First apply Lemma \ref{BTstep3} to obtain $\sigma_2$ that satisfies the desired conclusion for $d_2 = d_1/4$. We will next assume $|\sigma_1|\geq d_2n\|A\|^{-2}$ and $d_1d_2n\|A\|^{-4} \geq 1$.  Let $d_3 = d_2^2$.  Choose $\tilde\sigma_2\subset\sigma_2$ with the largest possible cardinality that satisfies the implicit bound $|\stilde \sigma_2| \leq d_1(|\sigma_1|+|\stilde\sigma_2|)/\|A\|^2$ (such sets exist; the empty set is one such example). Define $n' = |\sigma_1| + |\stilde\sigma_2|$ and note that
\[
n'\geq |\sigma_1| \geq d_2\frac{n}{\|A\|^2} .
\]
Since $|\sigma_2|\geq d_2n/\|A\|^2$, the set $\sigma_2$ is sufficiently large to have allowed us to choose $\stilde \sigma_2$ such that
\begin{align*}
|\stilde \sigma_2| \geq \lfloor d_1d_2n/\|A\|^4 \rfloor &\geq (d_1d_2n/\|A\|^4)/2\quad(\text{because }d_1d_2n/\|A\|^4\geq 1)\\
 &\geq d_3n/\|A\|^{4}.
\end{align*}
Now we define $V=\langle e_j : j \in \sigma_1 \cup \stilde \sigma_2 \rangle$ and $A'=A{|_V}$. An application of Lemma \ref{BTstep3} to the matrix $A'$ and the set $\stilde \sigma_{2}$, yields a subset $\tau_1 \subset \sigma_1$ such that 
\[
|\tau_1| \geq d_2n'/\|A'\|^2 \geq d_2^2n/\|A\|^4\geq d_3n/\|A\|^{4},
\]
and for any choice of scalars $\{a_{j}\}_{j \in \tau_1 \cup \stilde \sigma_{2}}$, we have
\begin{align}
    \norm{\sum_{j \in \tau_1 \cup \stilde \sigma_{2}}a_jAe_j} \ge \frac{1}{16\sqrt{2}} \left(\sum_{j \in \tau_1}|a_j|^2\right)^{1/2}. \label{eq4}
\end{align}
As $\sigma_2$ was provided by Lemma \ref{BTstep3} and  $\tau_1\subset \sigma_1$, $\stilde\sigma_2\subset\sigma_2$ we also have
\begin{align}
    \norm{\sum_{j \in \tau_1 \cup \stilde \sigma_{2}}a_jAe_j} \ge \frac{1}{16\sqrt{2}} \left(\sum_{j \in \stilde\sigma_2}|a_j|^2\right)^{1/2}. \label{eq5}
\end{align}
Recall that for any two non-negative numbers $x$, $y$ we have $\max\{x,y\} \geq (1/\sqrt2)(x^2+y^2)^{1/2}$, which in conjunction with \eqref{eq4} and \eqref{eq5} yields the desired inequality.
\end{proof}

The above theorem guarantees the existence of subsets of columns that can be ``stitched" together without violating the minimal stretch property.

\subsection{The proof of Theorem \ref{finalresultBTstatement}}
Recall that an $m$-dimensional subspace $U$ of $\mathbb{R}^n$ is a $\lambda$-quadratic convex combination of $U_{\sigma_1}$ and $U_{\sigma_2}$, where $\lambda\in[0,1]$ and $\sigma_1 = \{i_1<\cdots<i_m\}$, $\sigma_2 = \{j_1<\cdots<j_m\}$ are disjoint subsets of $\{1,\ldots,n\}$, such that $U$ is spanned by the orthonormal sequence $u_k = \lambda^{1/2}e_{i_k}+(1-\lambda)^{1/2}e_{j_k}$, $k=1,\ldots,m$.

In order to prove Theorem \ref{finalresultBTstatement}, we will combine Theorem \ref{BTbigthm} with a continuous traversal between disjoint subsets of the basis via quadratic convex combinations.

\begin{lem} \label{columnswitch}
 Let $A : [a,b] \to M_{n\times n}$ be a continuous matrix function and let $\sigma_1$, $\sigma_2 \subset \{1,\dotsc,n\}$ such that $|\sigma_1| = |\sigma_2|$ and $\sigma_1 \cap \sigma_2 = \emptyset$. Let $U = \langle e_i : i \in \sigma_1\cup \sigma_2\rangle$ and suppose there exists a $c>0$ such that $m_{A(t)|_U} \ge c$ for all $t \in [a,b]$. Let $\{U(t)\}_{t\in[a,b]}$ be such that, for each $t\in[a,b]$, $U(t)$ is the $(1 - (t-a)/(b-a))$-quadratic convex combination of $U_{\sigma_1}$ and $U_{\sigma_2}$. Then $\{U(t)\}_{t\in[a,b]}$ is a continuous choice of subspaces of $\mathbb{R}^n$ such that $U(a) = U_{\sigma_1}$, $U(b) = U_{\sigma_2}$, and $m_{A(t)|_{U(t)}} \ge c$ for all $t \in [a,b]$.
\end{lem}

\begin{proof}
Write $\sigma_1 = \{i_1<\cdots<i_m\}$, $\sigma_2 = \{j_1<\cdots<j_m\}$ so that for every $t \in [a,b]$,
\[u_k(t) = \Big(1-\frac{t-a}{b-a}\Big)^{1/2}e_{i_k} + \Big(\frac{t-a}{b-a}\Big)^{1/2}e_{j_k},\;k=1,\ldots,m\]
is an orthonormal basis for $U(t)$. As this is a continuous choice of orthonormal basis, by Lemma \ref{continuous gram-schmidt}, $\{U(t)\}_{t\in[a,b]}$ is a continuous choice of subspaces and clearly $U(a) = U_{\sigma_1}$, $U(b) = U_{\sigma_2}$.

To complete the proof, for $t\in[a,b]$ note that
\begin{align*}
    m_{A(t)|_{U(t)}} &= \inf\{\|A(t)x\| : x \in U(t), \|x\| = 1\}\\
    &\ge \inf\{\|A(t)x\| : x \in U, \|x\| = 1\}\quad\text{(as $U(t)\subset U$)}\\
    &= m_{A(t)|_U}\ge c. \qedhere
\end{align*}
\end{proof}

We now have all the tools at our disposal to prove Theorem \ref{finalresultBTstatement}, which is restated for convenience.

\begin{thm} \label{finalresultBT}
There exist universal constants $c = 1/33$ and $d = 2^{-15}/100$ such that for all continuous matrix functions $A: \mathbb{R} \rightarrow M_{n \times n}$ with the property that $\|A(t)e_i\|=1$ for all $t \in \mathbb{R}$ and $ 1 \le i \le n$, there exists a continuous family of $m$-dimensional subspaces $\{U(t)\}_{t \in \mathbb R}$ of $\mathbb{R}^n$ with $m \ge dn/\Lambda^4$ where $\Lambda = \sup_{t\in\mathbb{R}} \|A(t)\|$ such that $\|A(t)v\| \geq c \|v\|$ for every $t \in \mathbb{R}$ and every $v \in U(t)$. Furthermore, each subspace $U(t)$ is a quadratic convex combination of disjoint basis vectors.
\end{thm}

\begin{proof}
Note that $c = 1/33$ is a perturbation of $1/32$ and $d = d_3$, which are from the conclusion of Theorem \ref{BTbigthm} and Remark \ref{actual constants}. If $d_1d_2n/\Lambda^4 \leq 1$, then we may simply select $U(t) = \langle e_1\rangle$, for all $t\in\mathbb{R}$. Hence, we may assume $d_1d_2n/\Lambda^4 \geq 1$.  Note that
\[\Big\lceil d_2\frac{n}{\Lambda^2}\Big\rceil =  \Big\lceil \frac{1}{4}d_1\frac{n}{\Lambda^2}\Big\rceil \leq  d_1\frac{n}{\Lambda^2},\]
since $d_1n/\Lambda^2\geq d_1d_2n/\Lambda^4 \geq 1$. Furthermore,
\begin{align}
\Big\lceil d_3 \frac{n}{\Lambda^4}\Big\rceil &\leq d_3 \frac{n}{\Lambda^4} + 1 \leq \frac{d_2}{2\Lambda^2}\Big(d_2 \frac{n}{\Lambda^2}\Big) + \frac{d_1}{\Lambda^2}\Big(d_2\frac{n}{\Lambda^2}\Big)\leq 2d_1\Big(d_2\frac{n}{\Lambda^2}\Big)<\frac{1}{3}\Big\lceil d_2\frac{n}{\Lambda^2}\Big\rceil.\label{onesixth}
\end{align}

By Lemma \ref{compactification} we can find an increasing sequence of points $(t_i)_{i\in \mathbb{Z}}$ such that for all $i\in\mathbb{Z}$ and $t$ in $[t_{i-1},t_{i+1}]$, $\|A(t_i) - A(t)\| \leq \varepsilon = 1/32-1/33$.  Moreover, by Theorem \ref{BTbigthm} applied to the matrix $A(t_0)$ and the empty set, there exists a subset $\sigma_0\subset\{1,\ldots,n\}$ with $|\sigma_0| =  \lceil d_2n/\Lambda^2\rceil \leq d_1n/\Lambda^2$ such that for any choice of coefficients $\{a_j\}_{j\in\sigma_0}$ we have
\[\Big\|\sum_{j\in\sigma_0}a_jA(t_0)e_j\Big\| \geq \frac{1}{16\sqrt{2}}\Big(\sum_{j\in\sigma_0}|a_j|^2\Big)^{1/2} \geq \frac{1}{32}\Big(\sum_{j\in\sigma_0}|a_j|^2\Big)^{1/2}.\] 
By using $\|A(t_0) - A(t)\| \leq \varepsilon = 1/32-1/33$ we obtain that if $U = \langle e_j:j\in \sigma_0\rangle$, then $m_{A(t)|_U} \geq 1/33$ for all $t\in[t_{-1},t_1]$ (see the last part of Lemma \ref{compactification}).

We now focus on the interval $[t_0, \infty)$ as a symmetric argument can be applied to $(- \infty, t_0]$. In reality, an extra stitching needs to be performed at $t_0$ to concatenate the two solutions. We omit this as it is essentially contained in the ensuing argument.

Iteratively applying Theorem \ref{BTbigthm}, for $i \ge 0$, we can obtain the sets with following properties:
\begin{enumerate}[leftmargin=19pt,label=(\roman*)]

    \item \label{set1} $\sigma_{i+1}\subset\{1,\ldots,n\}\setminus\sigma_i$  with  $|\sigma_{i+1}| = \lceil d_2n/{\Lambda^2}\rceil \leq d_1 n/{\Lambda^{2}}$ such that, for $U = \langle e_j:j\in \sigma_{i+1}\rangle$, we have $m_{A(t)|_U} \ge 1/33$ for $t \in [t_i,t_{i+2}]$.
    
    \item \label{set2} $\dtilde \sigma_{i+1}\subset\sigma_{i+1}$ and $\stilde \tau_i\subset\sigma_i$  with $|\dtilde \sigma_{i+1}| = |\stilde \tau_{i}| = \lceil d_3 n/\Lambda^4 \rceil$ such that, for $U=\langle e_j : j \in \stilde \tau_i \cup \dtilde \sigma_{i+1} \rangle$, we have $m_{A(t)|_U} \ge 1/33$ for $t \in [t_i,t_{i+1}]$.

\end{enumerate}
Indeed, suppose $\sigma_i$  has been constructed with $|\sigma_i| = \lceil d_2n/{\Lambda^2}\rceil \leq d_1 n/{\Lambda^{2}}$ such that, for $U = \langle e_j:j\in \sigma_{i}\rangle$, we have $m_{A(t)|_U} \ge 1/33$ for $t \in [t_i,t_{i+2}]$. To obtain the sets $\sigma_{i+1}$, $\stilde\tau_i$ and $\dtilde\sigma_{i+1}$, we apply Theorem \ref{BTbigthm} to the matrix $A(t_{i+1})$ and the set $\sigma_i$. Because $|\sigma_i|\leq d_1n/\Lambda^2$, we first get $\sigma_{i+1}\subset\{1,\ldots,n\}$ with $|\sigma_{i+1}| \geq d_2n/\Lambda^2$ (which, after truncating, we may assume satisfies $|\sigma_{i+1}| = \lceil d_2n/\Lambda^2\rceil$) such that, for any choice of coefficients $\{a_j\}_{j\in\sigma_i\cup\sigma_{i+1}}$,
\[\Big\|\sum_{j\in\sigma_i\cup\sigma_{i+1}}a_jA(t_{i+1})e_j\Big\| \geq \frac{1}{32}\Big(\sum_{j\in\sigma_{i+1}}|a_j|^2\Big)^{1/2}.\] 
Using $\|A(t_{i+1}) - A(t)\| \leq \varepsilon = 1/32-1/33$ we obtain that,  if $U = \langle e_j:j\in \sigma_{i+1}\rangle$, then $m_{A(t)|_U} \geq 1/33$ for all $t\in[t_{i},t_{i+2}]$.  Since $|\sigma_i| = \lceil d_2n/\Lambda^2\rceil \geq d_2n/\Lambda^2$, the second part of Theorem \ref{BTbigthm} yields $\stilde \tau_i \subset \sigma_i$ \text{ and } $\dtilde \sigma_{i+1}\subset\sigma_{i+1}$ with $|\stilde\tau_i|,|\dtilde \sigma_{i+1}|\geq d_3n/\Lambda^4$ (which after truncating we may assume $|\stilde\tau_i| = |\dtilde \sigma_{i+1}| = \lceil d_3n/\Lambda^4\rceil$) such that for any choice of coefficients $\{a_j\}_{j\in \stilde\tau_i\cup\dtilde\sigma_{i+1}}$ we have
\[\Big\|\sum_{j\in\stilde\tau_i\cup\dtilde\sigma_{i+1}}a_jA(t_{i+1})e_j\Big\| \geq \frac{1}{32}\Big(\sum_{j\in\stilde\tau_i\cup\dtilde\sigma_{i+1}}|a_j|^2\Big)^{1/2}.\] 
By using $\|A(t_{i+1}) - A(t)\| \leq \varepsilon = 1/32-1/33$ we obtain that,  if $U = \langle e_j:j\in \stilde\tau_i\cup\dtilde\sigma_{i+1}\rangle$, then $m_{A(t)|_U} \geq 1/33$ for all $t\in[t_{i},t_{i+1}]$.

To construct our continuous collection of subspaces, consider the interval $[t_i, t_{i+1}]$.  Using $\langle e_j : j \in \stilde \tau_i \rangle$ at $t_i$ and $\langle e_j : j \in \dtilde \sigma_{i+1} \rangle$ at $t_{i+1}$, we will invoke Lemma \ref{columnswitch} to construct a continuous collection of subspaces $\{U(t)\}_{t \in [t_i,t_{i+1}]}$, that has dimension $\lceil d_3 n/\Lambda^4\rceil$ for every $i\geq 0$. A diagram of the construction has been provided below for the reader's convenience:

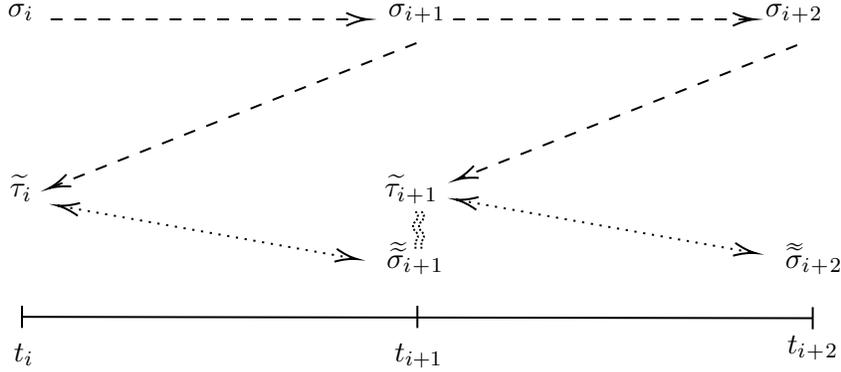
\begin{center}
\vspace{0.2in}

\tikzset{every picture/.style={line width=0.75pt}} 

\begin{tikzpicture}[x=0.75pt,y=0.75pt,yscale=-1,xscale=1]

\draw    (121.54,198.49) -- (520.34,199.17) ;
\draw [shift={(520.34,199.17)}, rotate = 180.1] [color={rgb, 255:red, 0; green, 0; blue, 0 }  ][line width=0.75]    (0,5.59) -- (0,-5.59)   ;
\draw [shift={(121.54,198.49)}, rotate = 180.1] [color={rgb, 255:red, 0; green, 0; blue, 0 }  ][line width=0.75]    (0,5.59) -- (0,-5.59)   ;
\draw    (320.95,192.93) -- (320.95,204.11) ;
\draw  [dash pattern={on 4.5pt off 4.5pt}]  (136,48.33) -- (294,48.33) ;
\draw [shift={(296,48.33)}, rotate = 180] [color={rgb, 255:red, 0; green, 0; blue, 0 }  ][line width=0.75]    (10.93,-3.29) .. controls (6.95,-1.4) and (3.31,-0.3) .. (0,0) .. controls (3.31,0.3) and (6.95,1.4) .. (10.93,3.29)   ;
\draw  [dash pattern={on 4.5pt off 4.5pt}]  (320.5,60.33) -- (136.86,132.93) ;
\draw [shift={(135,133.67)}, rotate = 338.43] [color={rgb, 255:red, 0; green, 0; blue, 0 }  ][line width=0.75]    (10.93,-3.29) .. controls (6.95,-1.4) and (3.31,-0.3) .. (0,0) .. controls (3.31,0.3) and (6.95,1.4) .. (10.93,3.29)   ;
\draw  [dash pattern={on 4.5pt off 4.5pt}]  (339,48.33) -- (489,48.33) ;
\draw [shift={(491,48.33)}, rotate = 180] [color={rgb, 255:red, 0; green, 0; blue, 0 }  ][line width=0.75]    (10.93,-3.29) .. controls (6.95,-1.4) and (3.31,-0.3) .. (0,0) .. controls (3.31,0.3) and (6.95,1.4) .. (10.93,3.29)   ;
\draw  [dash pattern={on 4.5pt off 4.5pt}]  (512.5,61.33) -- (342.36,128.6) ;
\draw [shift={(340.5,129.33)}, rotate = 338.43] [color={rgb, 255:red, 0; green, 0; blue, 0 }  ][line width=0.75]    (10.93,-3.29) .. controls (6.95,-1.4) and (3.31,-0.3) .. (0,0) .. controls (3.31,0.3) and (6.95,1.4) .. (10.93,3.29)   ;
\draw  [dash pattern={on 0.84pt off 2.51pt}]  (141.47,142.69) -- (288.53,168.98) ;
\draw [shift={(290.5,169.33)}, rotate = 190.14] [color={rgb, 255:red, 0; green, 0; blue, 0 }  ][line width=0.75]    (10.93,-3.29) .. controls (6.95,-1.4) and (3.31,-0.3) .. (0,0) .. controls (3.31,0.3) and (6.95,1.4) .. (10.93,3.29)   ;
\draw [shift={(139.5,142.33)}, rotate = 10.14] [color={rgb, 255:red, 0; green, 0; blue, 0 }  ][line width=0.75]    (10.93,-3.29) .. controls (6.95,-1.4) and (3.31,-0.3) .. (0,0) .. controls (3.31,0.3) and (6.95,1.4) .. (10.93,3.29)   ;
\draw  [dash pattern={on 0.84pt off 2.51pt}]  (342.47,139.69) -- (489.53,165.98) ;
\draw [shift={(491.5,166.33)}, rotate = 190.14] [color={rgb, 255:red, 0; green, 0; blue, 0 }  ][line width=0.75]    (10.93,-3.29) .. controls (6.95,-1.4) and (3.31,-0.3) .. (0,0) .. controls (3.31,0.3) and (6.95,1.4) .. (10.93,3.29)   ;
\draw [shift={(340.5,139.33)}, rotate = 10.14] [color={rgb, 255:red, 0; green, 0; blue, 0 }  ][line width=0.75]    (10.93,-3.29) .. controls (6.95,-1.4) and (3.31,-0.3) .. (0,0) .. controls (3.31,0.3) and (6.95,1.4) .. (10.93,3.29)   ;
\draw  [dash pattern={on 0.75pt off 0.75pt}]  (323,145.33) .. controls (324.67,147) and (324.67,148.66) .. (323,150.33) .. controls (321.33,152) and (321.33,153.66) .. (323,155.33) .. controls (324.67,157) and (324.67,158.66) .. (323,160.33) -- (323,164.33) -- (323,164.33)(320,145.33) .. controls (321.67,147) and (321.67,148.66) .. (320,150.33) .. controls (318.33,152) and (318.33,153.66) .. (320,155.33) .. controls (321.67,157) and (321.67,158.66) .. (320,160.33) -- (320,164.33) -- (320,164.33) ;

\draw (112.96,39) node [anchor=north west][inner sep=0.75pt]    {$\sigma _{i}$};
\draw (304.78,39) node [anchor=north west][inner sep=0.75pt]    {$\sigma _{i}{}_{+}{}_{1}$};
\draw (495.02,39) node [anchor=north west][inner sep=0.75pt]    {$\sigma _{i}{}_{+}{}_{2}$};
\draw (114,130) node [anchor=north west][inner sep=0.75pt]    {$\tau _{i}$};
\draw (303,130) node [anchor=north west][inner sep=0.75pt]    {$\tau _{i}{}_{+}{}_{1}$};
\draw (304,166) node [anchor=north west][inner sep=0.75pt]    {$\sigma _{i}{}_{+}{}_{1}$};
\draw (505,166) node [anchor=north west][inner sep=0.75pt]    {$\sigma _{i}{}_{+}{}_{2}$};
\draw (114,125) node [anchor=north west][inner sep=0.75pt]    {$\widetilde{\ \ }$};
\draw (304,125) node [anchor=north west][inner sep=0.75pt]    {$\widetilde{\ \ }$};
\draw (305,162.4) node [anchor=north west][inner sep=0.75pt]    {$\widetilde{\ \ }$};
\draw (305,159.4) node [anchor=north west][inner sep=0.75pt]    {$\widetilde{\ \ }$};
\draw (505,160) node [anchor=north west][inner sep=0.75pt]    {$\widetilde{\ \ }$};
\draw (505,162) node [anchor=north west][inner sep=0.75pt]    {$\widetilde{\ \ }$};
\draw (116,210.4) node [anchor=north west][inner sep=0.75pt]    {$t_{i}$};
\draw (308,210.4) node [anchor=north west][inner sep=0.75pt]    {$t_{i+1}$};
\draw (506,206.4) node [anchor=north west][inner sep=0.75pt]    {$t_{i}{}_{+}{}_{2}$};

\end{tikzpicture}
\vspace{0.2in}
\end{center}

As indicated on the diagram there is no reason to suppose $\stilde \tau_{i+1} = \dtilde \sigma_{i+1}$, so we must devise a way to ``stitch'' together the collection of subspaces at each endpoint, in order to construct a continuous collection of subspaces $\{U(t)\}_{t \in \mathbb{R}}$. We choose an $\eta_i \in (0, \frac{t_{i+1}-t_i}{2})$ and allow for another instance of ``stitching'' in the interval $[t_{i+1} - \eta_i, t_{i+1})$. This allows for us to continuously switch between $[t_i,t_{i+1}]$ and $[t_{i+1},t_{i+2}]$.

It is necessary to observe here that both $\stilde \tau_{i+1}$ and $\dtilde \sigma_{i+1}$ are subsets of $\sigma_{i+1}$. Thus the minimal stretch property holds for $U=\langle e_j : j \in \stilde \tau_{i+1} \cup \dtilde \sigma_{i+1} \rangle.$ We would like to switch from $\langle e_j : j \in \dtilde \sigma_{i+1} \rangle$ at $t_{i-1} -\eta_i$ to $\langle e_j:j \in \stilde \tau_{i+1} \rangle$ at $t_{i+1}$. To do this, it suffices to choose a new subset $\xi_{i+1}\subset \sigma_{i+1}\setminus(\stilde\tau_{i+1}\cup\dtilde\sigma_{i+1})$ with $|\xi_{i+1}| = |\stilde\tau_{i+1}| = |\dtilde\sigma_{i+1}| = \lceil d_3n/\Lambda^4\rceil$. This is possible because
\[|\stilde\tau_{i+1}\cup\dtilde\sigma_{i+1}| \leq 2\Big\lceil d_3\frac{n}{\Lambda^4}\Big\rceil  < \frac{2}{3}\Big\lceil d_2\frac{n}{\Lambda^2}\Big\rceil = \frac{2}{3}|\sigma_{i+1}|\quad\text{(by \eqref{onesixth}).}\]

We can now switch between the following subspaces for each $[t_i, t_{i+1}]$:
\begin{align*}
    \text{At } t_i &: \text{Begin with } U(t_i)= \langle e_j : j \in \stilde \tau_i \rangle.\\
    \text{At } t_{i+1} - \eta_i &: \text{Switch to  }U(t_{i+1}-\eta_i)= \langle e_j : j \in \dtilde \sigma_{i+1} \rangle \text{ using Lemma \ref{columnswitch}.}\\
        \text{At } t_{i+1} - \eta_i/2 &: \text{Switch to  }U(t_{i+1}-\eta_i/2)= \langle e_j : j \in \xi_i \rangle \text{ using Lemma \ref{columnswitch}.}\\
    \text{At } t_{i+1} &: \text{Switch to  }U(t_{i+1})= \langle e_j : j \in \stilde \tau_{i+1} \rangle \text{ using Lemma \ref{columnswitch}.} \qedhere
\end{align*}
\end{proof}

\section*{Acknowledgements}

The fourth and fifth authors are greatly indebted to Professors Alfonso Gracia-Saz (1976-2021) and Inder Bir Passi (1939-2021), without whose unfaltering support and encouragement the authors would never have considered continued studies in mathematics and to whom far more is owed than can be expressed. This paper is dedicated to their memories.

\bibliographystyle{plain}

\bibliography{bibliography}

\end{document}